\documentclass[12pt]{amsart}

\usepackage{amsthm,amsfonts,amsmath,amssymb,latexsym,epsfig,mathrsfs,yfonts,marvosym}
\usepackage[usenames]{color}
\usepackage{graphicx}
\usepackage[all]{xy}
\usepackage{epsfig}
\usepackage{epic}
\usepackage{eepic}
\usepackage{hyperref}
\usepackage{dsfont}\let\mathbb\mathds
\usepackage[english]{babel}

\DeclareMathAlphabet\oldmathcal{OMS}        {cmsy}{b}{n}
\SetMathAlphabet    \oldmathcal{normal}{OMS}{cmsy}{m}{n}
\DeclareMathAlphabet\oldmathbcal{OMS}       {cmsy}{b}{n}
 
\usepackage{eucal}

\usepackage[active]{srcltx}

\newtheorem{theorem}{Theorem}[section]
\newtheorem{lemma}[theorem]{Lemma}
\newtheorem{proposition}[theorem]{Proposition}
\newtheorem{corollary}[theorem]{Corollary}
\newtheorem{definition}[theorem]{Definition}

\newtheorem{question}[theorem]{Question}

\newenvironment{example}{\medskip \refstepcounter{theorem}
\noindent  {\bf Example \thetheorem}.\rm}{\,}
\newenvironment{remark}{\medskip \refstepcounter{theorem}
\newcommand     {\comment}[1]   {}
\newcommand{\mute}[2] {}
\newcommand     {\printname}[1] {}

\noindent  {\bf Remark \thetheorem}.\rm}{\,}
\newtheorem*{ack}{Acknowledgements}
\newtheorem{prop}[theorem]{Proposition}

\def\<{\langle}

\def\>{\rangle}

\def\BOne{{\mathchoice {\rm 1\mskip-4mu l} {\rm 1\mskip-4mu l}
                          {\rm 1\mskip-4.5mu l} {\rm 1\mskip-5mu l}}}

\def\fract#1#2{\raise4pt\hbox{$ #1 \atop #2 $}}
\def\decdnar#1{\phantom{\hbox{$\scriptstyle{#1}$}}
\left\downarrow\vbox{\vskip15pt\hbox{$\scriptstyle{#1}$}}\right.}

\def\bbc{{\mathbb C}}

\def\bbn{{\mathbb N}}

\def\bbp{{\mathbb P}}
\def\bbq{{\mathbb Q}}
\def\bbr{{\mathbb R}}

\def\bbz{{\mathbb Z}}

\def\gra{\alpha}
\def\grb{\beta}

\def\grd{\delta}

\def\grg{\gamma}

\def\grk{\kappa}
\def\grl{\lambda}

\def\gro{\omega}

\def\grr{\rho}

\def\grt{\tau}

\def\grz{\zeta}
\def\grD{\Delta}

\def\grG{\Gamma}

\def\grS{\Sigma}

\def\bfv{{\bf v}}
\def\bfw{{\bf w}}

\def\cala{{\mathcal A}}

\def\calc{{\mathcal C}}
\def\calo{{\mathcal O}}

\def\cald{{\mathcal D}}

\def\calf{{\mathcal F}}

\def\cali{{\mathcal I}}

\def\calm{{\mathcal M}}

\def\calo{{\mathcal O}}

\def\calr{{\mathcal R}}
\def\cals{{\oldmathcal S}}

\def\la#1{\hbox to #1pc{\leftarrowfill}}
\def\ra#1{\hbox to #1pc{\rightarrowfill}}
\def\calz{{\oldmathcal Z}}

\def\ga{{\mathfrak a}}

\def\ge{{\mathfrak e}}

\def\gm{{\mathfrak m}}
\def\gn{{\mathfrak n}}
\def\go{{\mathfrak o}}

\def\gt{{\mathfrak t}}
\def\gu{{\mathfrak u}}

\def\gz{{\mathfrak z}}

\def\gA{{\mathfrak A}}
\def\gB{{\mathfrak B}}
\def\gC{{\mathfrak C}}

\def\gH{{\mathfrak H}}

\def\gJ{{\mathfrak J}}

\def\gL{{\mathfrak L}}
\def\gM{{\mathfrak M}}

\def\gQ{{\mathfrak Q}}
\def\gR{{\mathfrak R}}

\def\hook{\mathbin{\hbox to 6pt{%
                 \vrule height0.4pt width5pt depth0pt
                 \kern-.4pt
                 \vrule height6pt width0.4pt depth0pt\hss}}}

\def\cJ{\hat{J}}

\def\tphi{\tilde{\phi}}

\def\th{\tilde{h}}

\def\12{\xi_{k_1,k_2}}
\def\m5{M^5_{k_1,k_2}}

\begin{document}

\title{Extremal Sasakian Geometry on $S^3$-bundles over Riemann Surfaces}

\author{Charles P. Boyer and Christina W. T{\o}nnesen-Friedman}\thanks{Both authors were partially supported by grants from the Simons Foundation, CPB by (\#245002) and CWT-F by (\#208799)}
\address{Charles P. Boyer, Department of Mathematics and Statistics,
University of New Mexico, Albuquerque, NM 87131.}
\email{cboyer@math.unm.edu} 
\address{Christina W. T{\o}nnesen-Friedman, Department of Mathematics, Union
College, Schenectady, New York 12308, USA } \email{tonnesec@union.edu}

\keywords{Extremal Sasakian metrics, extremal K\"ahler metrics, join construction}

\subjclass[2000]{Primary: 53D42; Secondary:  53C25}

\maketitle

\markboth{Extremal Sasakian Geometry on $S^3$-bundles}{Charles P. Boyer and Christina W. T{\o}nnesen-Friedman}

\begin{abstract}
In this paper we study the Sasakian geometry on $S^3$-bundles over a Riemann surface $\grS_g$ of genus $g>0$ with emphasis on extremal Sasaki metrics. We prove the existence of a countably infinite number of inequivalent contact structures on the total space of such bundles that admit 2-dimensional Sasaki cones each with a Sasaki metric of constant scalar curvature (CSC). This CSC Sasaki metric is most often irregular. We further study the extremal subset $\ge$ in the Sasaki cone showing that if $0<g\leq 4$ it exhausts the entire cone. Examples are given where exhaustion fails.

\end{abstract}

\tableofcontents

\section{Introduction}
The purpose of this paper is to study the existence and non-existence of extremal Sasaki metrics on $S^3$-bundles over a Riemann surface of genus $g\geq 1$. It can be considered as a continuation of \cite{BoTo11} where the genus $g=1$ case was treated for the case of the trivial product structure $T^2\times S^3$. In this sequel we also study the non-trivial $S^3$-bundle over $T^2$. The genus $0$ case, however, will not be treated here. It has been studied in \cite{Boy10b,Leg10,BoPa10}. 

The problem of finding the ``best'' K\"ahler metrics in a given K\"ahler class on a K\"ahler manifold has a long and well developed history beginning with Calabi \cite{Cal82}. He introduced the notion of an {\it extremal K\"ahler metric} as a critical point of the $L^2$-norm squared of the scalar curvature of a K\"ahler metric, and showed that extremal K\"ahler metrics are precisely those such that the $(1,0)$-gradient of the scalar curvature is a holomorphic vector field. Particular cases are K\"ahler-Einstein metrics and more generally those of {\it constant scalar curvature (CSC)}. Since the works of Donaldson and Uhlenbeck-Yau \cite{Don85,UhYa86} the idea that such K\"ahler metric are related to the notion of stability (in some sense) of certain vector bundles has become a major area of study, see the recent survey \cite{PhSt10} and references therein.

On the other hand in the case of the odd dimensional version, namely Sasakian geometry, aside from the study of Sasaki-Einstein metrics, the study of extremal Sasaki metrics is much more recent \cite{BGS06,BGS07b,Boy10b,BoTo11,BoTo12b}. It is well known that a Sasaki metric has constant scalar curvature if and only if the transverse K\"ahler metric has constant scalar curvature. So for quasi-regular Sasaki metrics the existence of CSC Sasaki metrics coincides with the existence of CSC K\"ahler metrics on the quotient cyclic orbifolds $\calz$. In this case a result of Ross and Thomas \cite{RoTh11} says that the existence of a CSC Sasaki metric on $M$ implies the K-semistability of the corresponding polarized K\"ahler orbifold $(\calz,L)$. This result has been recently generalized to include irregular Sasaki metrics by Collins and Sz\'ekelyhidi \cite{CoSz12} which says that a CSC Sasaki metric implies the K-semistability of the corresponding K\"ahler metric on the cone $C(M)=M\times \bbr^+$. However, the cone metric on $C(M)$ does not generally have constant scalar curvature. 

The Sasaki-Einstein case is by far the best studied. If the first Chern class $c_1(\cald)$ of the contact bundle $\cald$ vanishes, then a Sasaki metric $g$ with contact bundle $\cald$ has constant scalar curvature if and only if $g$ is Sasaki-$\eta$-Einstein (see \cite{BGM06,BGS06} and references therein). So the analogues of Sasaki-Einstein and more generally Sasaki-$\eta$-Einstein metrics when $c_1(\cald)\neq 0$ are the CSC Sasaki metrics. This leads one to the following interesting questions:

\begin{question}\label{ques}
Given a Sasakian manifold $M$. When is there a ray of constant scalar curvature in the Sasaki cone? 
\end{question}

Of course it is quite easy to construct such examples of CSC Sasaki metrics. Just take a polarized K\"ahler manifold (or orbifold) $\calz$ with constant scalar curvature and use Boothby-Wang \cite{BoWa,BG00a} to construct the Sasaki $S^1$-bundle $M$ over it. Then $M$ will have a Sasaki metric of constant scalar curvature. However, suppose that the initial K\"ahler manifold (orbifold) does not have constant scalar curvature, can one deform in the Sasaki cone to obtain Sasaki metrics of constant scalar curvature? The first example showing that the latter was possible was given by the physicists Gauntlett, Martelli, Sparks, and Waldram \cite{GMSW04a}. Here one considers the blow-up of $\bbc\bbp^2$ at a point polarized by the anti-canonical line bundle, that is, $c_1$ is proportional to the K\"ahler class $[\gro]$. It is well known not to have a K\"ahler-Einstein (hence, CSC) metric. The Boothby-Wang construction here gives a Sasakian structure on $S^2\times S^3$ with $c_1(\cald)=0$ and a 3-dimensional Sasaki cone (the toric case). Nevertheless, in \cite{GMSW04a} it is shown by an explicit construction of the metric that there is an irregular Sasaki-Einstein metric in the Sasaki cone\footnote{Actually, they construct a countable infinity of such Sasaki-Einstein metrics one of which is the one described here. Some of these metrics are quasi-regular.}. This idea was then generalized to any toric contact structure of Reeb type with $c_1(\cald)=0$ by Futaki, Ono, and Wang \cite{FOW06}. They proved that for any such contact structure one can deform in the Sasaki cone to obtain a Sasaki-Einstein metric. Moreover, this Sasaki-Einstein metric is unique \cite{CFO07}. 

Much less is known when $c_1(\cald)\neq 0$. However, Legendre \cite{Leg10} has shown that for the toric structures on $S^3$-bundles over $S^2$ and certain quotients, one can always deform in the Sasaki cone to obtain a ray of Sasaki metrics of constant scalar curvature. Moreover, she shows that uniqueness fails, that is, there are examples with more than one ray of CSC Sasaki metrics.

A main result of the present paper is
\begin{theorem}\label{thm1}
Let $M$ be the total space of an $S^3$-bundle over a Riemann surface $\grS_g$ of genus $g\geq 1$. Then for each $k\in \bbz^+$ the 5-manifold $M$ admits a contact structure $\cald_k$ of Sasaki type consisting of $k$ 2-dimensional Sasaki cones $\grk(\cald_k,J_m)$ labelled by $m=0,\ldots,k-1$ each of which admits a unique ray of Sasaki metrics of constant scalar curvature such that the transverse K\"ahler structure admits a Hamilitonian 2-form. Moreover, if $M$ is the trivial bundle $\grS_g\times S^3$, the $k$ Sasaki cones belong to inequivalent $S^1$-equivariant contact structures.
\end{theorem}

We refer to the unique CSC ray given by this theorem as the {\it admissible CSC ray} and the corresponding Sasakian structures as {\it admissible Sasakian structures}.

\begin{remark}\label{irrrem}
When $m>0$ most admissible CSC Sasakian structures are irregular, that is not quasi-regular. For the trivial bundle $\grS_g\times S^3$ with $m=0$ the CSC Sasakian structure is regular.
\end{remark}

When the Sasaki cones belong to contact structures that are inequivalent as $S^1$-equivariant contact structures, we have the notion of a bouquet of Sasaki cones \cite{Boy10a,BoTo11}. So in the case of $\grS_g\times S^3$ our results give a $k+1$-bouquet $\gB_{k+1}(\cald_k)$, consisting of $k$ 2-dimensional Sasaki cones and one 1-dimensional Sasaki cone on each contact structure $\cald_k$, all of which have CSC rays. Furthermore, these are invariant under certain deformations of the transverse complex structure. We refer to Theorem \ref{sasbouquetthm} for details.

We also have the following
\begin{theorem}\label{sasexh}
Given any genus $g>0$ and non-negative integer $M$, there exists a positive integer $K_{g,M}$ such that for all integers $k\geq K_g$ the  two dimensional Sasaki cones $\grk(\cald_k,J_m)$ in the contact structure $\cald_k$ are exhausted by extremal Sasaki metrics for $m=0,\ldots,M$. In particular, if $0<g\leq 4$ all of the $k$ 2-dimensional Sasaki cones are exhausted by extremal Sasaki metrics for all $k\in\bbz^+$. Hence,  in the case of the trivial bundle $\grS_g\times S^3$ when $2\leq g\leq 4$ the entire $k+1$-bouquet $\gB_{k+1}(\cald_k)$ is exhausted by extremal Sasaki metrics.

\end{theorem} 

On the other hand if we fix $k\in \bbz^+$ and let the genus grow, we can lose extremality. To see this we consider only regular Sasakian structures since in this case we can apply the uniqueness of extremal K\"ahler metrics by Chen and Tian \cite{ChTi05}, which is not yet available in the orbifold setting. So for regular Sasaki metrics we do not need the admissibility assumption discussed above. However, we do know that for genus $g=2,...,19$ the regular ray in the Sasaki cones $\grk(\cald_k,J_m)$ for $m=0,\ldots,k-1$  admits an extremal representative (with non-constant scalar curvature when $m>0$). See Theorem \ref{deg>0ext}.

\begin{theorem}\label{deg>0intro}
For any choice of genus $g=20,21,...$ there exist at least one choice of $(k,m)$ with $m=1,\ldots,k-1$ such that the regular ray in the Sasaki cone $\grk(\cald_k,J_m)$ admits no extremal representative. 
\end{theorem}

Hence, in the general setting we are interested in how much of the Sasaki cone can be represented by extremal Sasaki metrics. Let $\ge(\cald,J)$ denote the subset of the Sasaki cone $\grk(\cald,J)$ that can be represented by extremal Sasaki metrics. We know from Theorem \ref{thm1} that the {\it extremal set} $\ge(\cald_k,J_m)$ has a unique ray of admissible CSC Sasaki metrics for every $k\in\bbz^+$ and every $m=0,\ldots,k-1$. Furthermore, our main construction in Section \ref{s4} begins by taking the join of a circle bundle over the Riemann surface $\grS_g$ and a weighted Sasaki 3-sphere with weight vector $\bfw=(w_1,w_2)$. It follows from this construction that the Sasaki cones $\grk(\cald_k,J_m)$ contain another extremal ray, which we call the $\bfw$ ray, defined by this weight vector which is not a CSC ray except for the special case when $w_1=w_2$. We also have

\begin{theorem}\label{connthm}
For the 2-dimensional Sasaki cones the admissible CSC ray and the $\bfw$ ray belong to the same connected component $\calc$ of $\ge(\cald,J)$.
\end{theorem}

We recall here from \cite{BGS06} that $\ge(\cald,J)$ is open in $\grk(\cald,J)$, so $\calc$ is an open conical subset in $\grk(\cald,J)$.

The outline of the paper is as follows: In Section \ref{comsurf} we briefly review what is known about complex ruled surfaces. We divide these into 3 types. Although generally the description of all complex structures on ruled surfaces is very complicated, the ones of importance to us are more tractable. They are essentially the K\"ahler structures that admit Hamiltonian Killing vector fields and those induced from a stable rank two vector bundle. We briefly describe the effects of the representation theory of the fundamental group of the Riemann surface on our setup. The main result of this section is a result that uses equivariant Gromov-Witten invariants to show that certain Hamiltonian vector fields are non-conjugate in the group of Hamiltonian isotopies. 

In Section \ref{s4} we present the setup of our Sasakian structures on $S^3$-bundles over a Riemann surface $\grS_g$. This entails the `join' of a circle bundle over $\grS_g$ with a weighted 3-sphere $S^3_\bfw$ with weight vector $\bfw=(w_1,w_2)$. This approach differs from, but is equivalent to, the one used in \cite{BoTo11}. It has the advantage of being more flexible when deforming the Sasakian structures in a Sasaki cone. In this way we treat the non-trivial $S^3$ on an equal footing as the trivial bundle which was not done in \cite{BoTo11}. In particular, we show that all of our two dimensional Sasaki cones have a unique regular ray of Sasakian structures which allows for the aforementioned equivalence with the approach in \cite{BoTo11}.

In Section \ref{sasfam} we describe various families of Sasakian structures which belong to the same isotopy class of contact structures. Some belong to the same CR structure, but with a different characteristic foliation. To obtain extremal representatives we must deform through distinct CR structures, but with the same characteristic foliation. A main result of this section shows that the first Chern class of the contact bundle distinguishes the distinct contact structures.

Finally, in Section \ref{admis} we use the method of Hamiltonian 2-forms developed in \cite{ApCaGa06,ACGT04,ACGT08} to construct our admissible extremal transverse K\"ahlerian structures. The extremal metrics are determined by a certain fourth order polynomial, and the detailed analysis of this polynomial gives the theorems listed above.

\begin{ack}
The authors would like to thank Vestislav Apostolov, David Calderbank, Fred Cohen, Thomas Friedrich, Paul Gauduchon, and Claude LeBrun for many helpful conversations.
\end{ack}

\section{Ruled Surfaces of Genus $g\geq 1$}\label{comsurf}
By a {\it ruled surface} we shall mean what is often called a geometrically ruled surface. It is well known (cf. \cite{McDSa}, pg. 203) that there are precisely two diffeomorphism types of ruled surfaces of a fixed genus $g$. These are the trivial bundle $\grS_g\times S^2$  and non-trivial $S^2$-bundles over $\grS_g$ denoted by $\grS_g\tilde{\times}S^2$, and they are distinguished by their second Stiefel-Whitney class $w_2$. 

\subsection{Complex Structures on $\Sigma_g \times S^2$ and $\grS_g\tilde{\times}S^2$}\label{complexsubsec}
It is well known \cite{Ati55,Ati57,BPV} that all complex structures on ruled surfaces arise by considering them as projectivizations of rank two holomorphic vector bundles $E$ over a Riemann surface $\grS_g$ of genus $g$. Thus if $(M,J)$ is a ruled surface of genus $g$ we can write it as  $(M,J) = {\mathbb P}(E) \rightarrow \Sigma_g$, where $\Sigma_g$ is equipped with a complex structure $J_\tau$ where $\tau\in \calm_g$, the moduli space of complex structures on $\grS_g$.

Nevertheless, the classification of complex structures on a ruled surface is quite complicated \cite{Maru70}, and the moduli space is generally non-Hausdorff. However, a versal deformation space does exist \cite{Suw69,Sei92}, and it will be convenient to divide the vector bundles $E$ into two types, the indecomposable bundles and the decomposable bundles that can be written as the sum of line bundles. In the latter case we can (and shall) by tensoring with a line bundle bring $E$ to the form $\calo\oplus L$ where $\calo$ denotes the trivial line bundle and $L$ is a line bundle of degree $n \in \bbz$. In this case one can think of the complex structures as arising from three sources, the moduli space of complex structures on $\grS_g$, the degree of $L$, and certain irreducible representations of the fundamental group of $\grS_g$. Let us further subdivide our bundles as:
\begin{enumerate}
\item $E$ is indecomposable \begin{enumerate}
\item $E$ is stable 
\item $E$ is non-stable
\end{enumerate}
\item $E= {\mathcal O} \oplus L$, where $L$ is a degree $0$ holomorphic line bundle on $\Sigma_g$ and ${\mathcal O}$ denotes the trivial (holomorphic) line bundle on $\Sigma_g$.
\item $E= {\mathcal O} \oplus L$, where $L$ is a  holomorphic line bundle on $\Sigma_g$ of non-zero degree $n$.
\end{enumerate}
For the rest of the paper we will refer to the above cases as cases (1), (1)(a), (1)(b), (2), and (3). 
Case (1)(a) and case (2)  will be described in detail in Section \ref{repr} via representation spaces. To get a simple example of case (1)(b), one could let $E$ be a non-trivial extension of the trivial line bundle. It should be mentioned that our division into types is not invariant under deformations. Indeed, two ruled surfaces are deformation equivalent if and only if they are homeomorphic \cite{Sei92}.


Note that if $(M,J)$ is diffeomorphic to $\Sigma_g\times S^2$ then $\deg(E)$ is even (see exercise (8) on page 38 in \cite{Bea96}) or equivalently ${\mathbb P}(E) \rightarrow \Sigma_g$ admits sections of even self-intersection while if $(M,J)$ is diffeomorphic to $\grS_g\tilde{\times}S^2$ then $\deg(E)$ is odd. Since we are dealing with projective bundles of the form $\bbp(E)$ we are free to tensor $E$ with any line bundle.  In particular, we can assume that $E$ has degree $0$ or $1$ in case (1).

We will also denote the complex manifolds occurring in case (2) and (3)  by $S_n$, where $n= \deg L$, and call them {\em pseudo-Hirzebruch surfaces}. We note that $S_{-n}$ is biholomorphic to $S_n$. Beware that unless $g=0$,  $n$ does not uniquely determine the complex structure since there are many different holomorphic line bundles of the same degree. If $n$ is even, $(M,J)$ is diffeomorphic to $\Sigma_g\times S^2$ and we write $n=2m$, while if $n$ is odd, $(M,J)$ is 
diffeomorphic to $\grS_g\tilde{\times}S^2$ and we write $n=2m+1$.

\begin{remark} If $\Sigma_g = T^2$,  then there is, up to biholomorphism, only one case of case (1) with $\deg(E) = 0$, namely (1)(b) and only one case of case (1) with
 $\deg(E)=1$, namely (1)(a). When $\Sigma_g = \bbc\bbp^1$, it is well-known that case (1) does not occur at all. 
\end{remark}

\bigskip

Assume now that $(M,J)$ is a pseudo-Hirzebruch surface $S_n$ with $n \geq 0$.
Let $E_n$ denote the zero  section,  of $M \rightarrow 
\Sigma_g$. Then $E_{n}\cdot E_{n} = n$ (where $n=0$ in case (2)). If $F$ denotes a fiber of the ruling $M \rightarrow 
\Sigma_g$, then $F\cdot F=0$, while $F \cdot E_{n} =1$.
Any real cohomology class in the two dimensional space 
$H^{2}(M, {\mathbb R})$ may be written as a linear combination of 
(the Poincare 
duals of) $E_{n}$ and $F$, 
$$
m_{1} PD(E_{n}) +m_{2}PD(F)\, .
$$
In particular, the K\"ahler cone ${\mathcal K}$ corresponds to $ m_{1} > 0, 
m_{2} > 0$ (see \cite{fuj92} or Lemma 1 in \cite{To-Fr98}). 

Consider the cohomology class $\alpha_{k_1,k_2} \in H^2(M,{\mathbb R})$ given by
\begin{equation}\label{KahclassA}
\alpha_{k_1,k_2} = k_1h + k_2PD(F),
\end{equation}
where $h=PD(E_0) = PD(E_n)-mPD(F)$ when $n=2m$ is even and $h=PD(E_1)=PD(E_n)-mPD(F)$ when $n=2m+1$ is odd. Then it easily follows from the above that $\alpha_{k_1,k_2}$ is a K\"ahler class if and only if $k_1>0$ and $k_2/k_1 > m$. 

The cohomology class in \eqref{KahclassA} is of course defined even if $J$ belongs to case (1) and we have the following general Lemma which is essentially the $rank\,E = 2$ case of 
Proposition 1 in \cite{fuj92} with a slight change of notation.

\begin{lemma}\cite{fuj92}\label{Kahclass}
\begin{itemize}
\item For any $(M,J)$ of case (1)(a) with $\deg(E)$ even,
$\alpha_{k_1,k_2}$ is a K\"ahler class if and only if $k_1,k_2 >0$.
\item For any $(M,J)$ of case (1)(a) with $\deg(E)$ odd,
$\alpha_{k_1,k_2}$ is a K\"ahler class if and only if $k_1 >0$ and $k_2 + \frac{k_1}{2}  > 0$.
\item For any $(M,J)$ of  case (1)(b),  let $\tilde{m}$ equal the (non-negative and well-defined) number $\max(E)-\deg(E)/2$, where $\max(E)$ denotes the maximal degree of a sub line bundle of $E$. 
\begin{itemize}
\item If $\deg(E)$ is even, $\alpha_{k_1,k_2}$ is a K\"ahler class if and only if
$k_1>0$ and $\frac{k_2}{k_1} >\tilde{m}$. 
\item If $\deg(E)$ is odd, $\alpha_{k_1,k_2}$ is a K\"ahler class if and only if
$k_1>0$ and $\frac{k_2}{k_1}+\frac{1}{2} >\tilde{m}$.
\end{itemize}
\item For any $(M,J)$ of case (2) with $\deg(E)$ even,
$\alpha_{k_1,k_2}$ is a K\"ahler class if and only if $k_1,k_2 >0$.
\item For any $(M,J)$ of case (3) above with   $n=2m$ or $n=2m+1$ and $m \in  {\mathbb Z}^+$, 
$\alpha_{k_1,k_2}$ is a K\"ahler class if and only if
$k_1>0$ and $\frac{k_2}{k_1} > m$. 
\end{itemize}
\end{lemma}

\bigskip

When $n=2m$ is even, the ruled surface is $\grS_g\times S^2$, and  $\alpha_{k_1,k_2}$ is represented by
the split symplectic $2$-form 
\begin{equation}\label{splitform}
\omega_{k_1,k_2} = k_1 \omega_0 + k_2 \omega_g,
\end{equation} 
where $\omega_g$ and $\omega_0$ are the standard area measures on $\Sigma_g $ and $S^2$, respectively. 
When $n=2m+1$ the ruled surface is $\Sigma_g \tilde{\times} S^2$ and it follows from a theorem of McDuff \cite{McD94} that $\alpha_{k_1,k_2}$ can be represented by a symplectic $2$-form that is non-degenerate on the fibers of the ruling and is compatible with the given orientation. We also denote this form by $\gro_{k_1,k_2}$, but it does not split as in Equation (\ref{splitform}).
\begin{remark}
Note that $\tilde{m}$ is well-defined in all of the cases of Lemma \ref{KahclassA}. In case (1)(a) it would be negative. In case (2) it would be zero. In case (3) it would be equal to $m$ when $n=2m$ and equal to $\frac{2m+1}{2}$ when $n=2m+1$. With that in mind, one could summarize the 
K\"ahler criterion for \eqref{KahclassA} as
$$k_1>0, \quad k_2> max(0,\tilde{m}k_1)$$
when the diffeomorphism type is  $\Sigma_g \times S^2$, and
$$k_1>0, \quad k_2+k_1/2> max(0,\tilde{m}k_1)$$
when the diffeomorphism type is $\grS_g\tilde{\times}S^2$.
\end{remark}

\subsection{Representation Spaces} \label{repr}
We begin by fixing a complex structure $\tau\in \calm_g$. We know from the work of Narasimhan and Seshadri \cite{NaSe65} that on a compact Riemann surface the unitary irreducible representations of the fundamental group into $PSU(2)$ correspond precisely to the stable rank two holomorphic vector bundles. This corresponds precisely to case (1)(a) above. Case (2) above are the polystable rank two holomorphic bundles that are not stable and correspond precisely to unitary reducible representations of  the fundamental group. Together cases (1)(a) and (2) are all the rank two holomorphic vector bundles over $\grS_g$ that are polystable. Moreover, these are precisely the ruled surfaces that admit a CSC K\"ahler metric \cite{ApTo06}. So the unitary projective representations of the fundamental group $\pi_1(\grS_g)$ correspond precisely to cases (1)(a) and (2) of subsection \ref{complexsubsec} above, (1)(a) are the irreducibles and (2) are the reducibles.

So here we consider representations of $\pi_1(\grS_g)$ in $PSU(2)\approx SO(3)$.  Following  \cite{AkMc90,Sav99} the representation space considered is $R(\grS_g)={\rm Hom}(\pi_1(\grS_g),PSU(2))$ which is given the compact-open topology with the discrete topology on $\pi_1(\grS_g)$ and the usual topology on $PSU(2)$ which we also identify with $\bbr\bbp^3$. Thus, for each $\grr\in R(\grS_g)$ we have the ruled surface $\grS_g\times_\grr\bbc\bbp^1$ obtained as a quotient of $D\times \bbc\bbp^1$ by the group $\pi_1(\grS_g)\times \grr(\pi_1(\grS_g))$ where $D$ is the unit disc in $\bbc$.  A representation $\grr\in {\rm Hom}(\pi_1(\grS_g),PSU(2))$ is {\it reducible} if it lies in a linear subspace in $\bbc^2\supset S^3$, otherwise, it is {\it irreducible}. Thus, $\grr$ is reducible if and only if it lies in a circle in $PSU(2)$. The group of K\"ahler automorphisms of $D\times \bbc\bbp^1$ is $PSU(2)$ acting on the second factor. The automorphisms that pass to the quotient $\grS_g\times_\grr\bbc\bbp^1$ are precisely the ones that commute with $\grr(\pi_1(\grS_g))$, and these will contain an $S^1$ if and only if $\grr$ is reducible. 

Recall that for $g\geq 2$, the fundamental group $\pi_1(\grS_g)$ has the well-known presentation given by generators $<a_1,b_1,\cdots,a_g,b_g>$ and one relation $\prod_{i=1}^g[a_i,b_i]=\BOne$.
So there is a natural map $\psi:R(\grS_g)\ra{1.6} SO(3)^{2g}$ defined by $\psi(\grr)=\bigl(\grr(a_1),\grr(b_1), \cdots,\grr(a_g),\grr(b_g)\bigr)$, and the image realizes the space $R(\grS_g)$ as a subspace of the product $(S^3)^{2g}\subset \bbc^{4g}$ satisfying the equation
$$\prod_{i=1}^g[\grr(a_i),\grr(b_i)]=\BOne.$$
Thus, $R(\grS_g)$ has the structure of a real algebraic variety of dimension $6g-3$. The smooth locus of $R(\grS_g)$ is precisely the open subset of irreducible representations, $R(\grS_g)^{irr}$. Actually, since conjugate representations are equivalent, we are more interested in the {\it character variety} $\calr(\grS_g)$ defined to be the quotient $\calr(\grS_g)=R(\grS_g)/SO(3)$ by conjugation in $SO(3)$. This action is free on the open subset $R(\grS_g)^{irr}$ (clearly, reducible and irreducible are preserved under conjugation), so the dimension of $\calr(\grS_g)$ is $6g-6$. 

Note that when $\grr$ is reducible, it must factor through $H_1(\grS_g,\bbz)\approx \bbz^{2g}$. Furthermore, there are two types of reducible representations, those mapping to an $S^1$ and those mapping to the identity $\BOne$. Clearly, the latter corresponds to the product $\grS_g\times \bbc\bbp^1$ whose automorphism group is $PSU(2)$. The subspace of reducible representations form the singular part of $\calr(\grS_g)$ and has dimension $2g$, and can be identified with the Picard group ${\rm Pic}^0(\grS_g)$ of line bundles on $\grS_g$ of degree $0$ which of course is just the Jacobian torus $T^{2g}$. If $\grr\neq \BOne$ the connected component of the centralizer of $\grr(\pi_1(\grS_g))$ in $PSU(2)$ is $S^1$ which is the connected component of automorphism group of the induced K\"ahler structure. This gives us a Hamiltonian Killing vector field. Summarizing we have

\begin{proposition}\label{hamred}
Let $M_g=\grS_g\times_\grr\bbc\bbp^1$ be a ruled surface arising from a projective unitary representation $\grr$ of $\pi_1(\grS_g)$. Then $M_g$ admits a holomorphic Hamiltonian circle action if and only if $\grr$ is reducible. Moreover, 
\begin{enumerate}
\item[(i)] $M_g$ is of type (1)(a) if and only if the representation $\grr$ is irreducible. 
\item[(ii)] $M_g$ is of type (2) if and only if $\grr$ is reducible. 
\end{enumerate}
\end{proposition}

\subsection{The Action of ${\rm Pic}^0(\grS_g)$ on Line Bundles of Degree $n$}\label{s2.3}
Consider the Picard group ${\rm Pic}(\grS_g)$ of holomorphic line bundles on $\grS_g$ together with its subgroup ${\rm Pic}^0(\grS_g)$ of holomorphic line bundles of degree zero. The group structure is given by tensor product, and we know that ${\rm Pic}^0(\grS_g)$ is a compact connected Abelian group of dimension $2g$ isomorphic to the Jacobian torus. Moreover, there is an exact sequence of Abelian groups
\begin{equation}\label{Picexact}
0\ra{1.8} {\rm Pic}^0(\grS_g)\ra{1.8} {\rm Pic}(\grS_g)\ra{1.8} \bbz\ra{1.8} 0.
\end{equation}
We denote by $\gL_n$ the set of holomorphic line bundles on $\grS_g$ of degree $n$. We can write ${\rm Pic}(\grS_g)$ as a direct sum ${\rm Pic}(\grS_g)=\bigoplus_n\gL_n$, and view it as a reducible ${\rm Pic}^0(\grS_g)$-module. The exact sequence (\ref{Picexact}) splits but there is no canonical splitting.
To understand the complex structures of case (3) of subsection \ref{complexsubsec} we note that the irreducible ${\rm Pic}^0(\grS_g)$-module $\gL_n$ is isomorphic as  ${\rm Pic}^0(\grS_g)$-modules to ${\rm Pic}^0(\grS_g)$ itself which in turn is identified with $T^{2g}$. So for $g\geq 2$ there is a $T^{2g}$'s worth of complex structures in $\gL_n$ with the complex structure $\grt$ on the Riemann surface $\grS_g$ fixed. Moreover, since an element of ${\rm Pic}^0(\grS_g)$ is determined by a reducible representation $\grr$, we can use $\grr$ to label an element of $\gL_n$. In the $g=1$ case the Jacobian variety is identified with $T^2$ itself, so all the complex structures coming from the Jacobian $T^2$ are equivalent \cite{Suw69}. 

\subsection{Extremal K\"ahler Metrics}\label{exKmet}

{\em Extremal K\"ahler metrics} are generalizations of constant
scalar curvature K\"ahler metrics:
Let $(M,J)$ be a compact complex manifold admitting at least one
K\"ahler metric. For a
particular K\"ahler class $\alpha$, let $\alpha^+$ denote the 
set of all K\"ahler forms in $\alpha$.

Calabi \cite{Cal82} suggested that one should look for extrema of the 
following functional $\Phi$ on $\alpha^+$:
\[
\Phi : \alpha^+ \rightarrow {\mathbb R}
\]
\[
\Phi(\omega) = \int_M s^2 d\mu,
\]
where $s$ is the scalar curvature and $d\mu$ is the volume form of the
K\"ahler metric corresponding to the K\"ahler form $\omega$.
Thus $\Phi$ is the square of the $L^2$-norm of the scalar curvature.

\bigskip

\begin{proposition}\cite{Cal82}
The K\"ahler form $\omega \in \alpha^+$ is an
extremal point of $\Phi$ if and only
if the gradient vector field $grad \, s$ is a holomorphic real vector field, that is
$\pounds_{grad \, s}J=0$. When this happens the metric $g$ corresponding to $\omega$ is
called an {\em extremal K\"ahler metric}.
\end{proposition}
Notice that if $\pounds_{grad \, s}J=0$, then $Jgrad\, s$ is a Hamiltonian Killing vector field inducing Hamiltonian isometries.

\bigskip

Let $(M,J) = \bbp(E) \rightarrow \Sigma_g$ for $g\geq 1$ as in Section \ref{complexsubsec}.
As mentioned previously
if $E$ is indecomposable, it follows from Lemma 1 in \cite{ACGT11} that $(M,J)$ admits no hamiltonian vector fields. This means that any extremal K\"ahler metric  on $(M,J)$ must be CSC.
In case (1)(a) $(M,J)$ admits a CSC K\"ahler metric in every K\"ahler class, constructed as the local product on $\bbp (E) \rightarrow \Sigma_g$.  More specifically, $(M,J)$ is the quotient ruled surface
$\Sigma_g \times_\rho \bbc\bbp^1$, where $\rho: \pi_1(\Sigma_g) \rightarrow PSU(2)$ is a irreducible unitary representation parameterized by the smooth locus $\calr(\grS_g)^{irr}$ of the character variety. The local product metric is inherited from the product metric on $D \times \bbc\bbp^1$. 
In case (1)(b) $(M,J)$ admits no extremal K\"ahler metrics \cite{ApTo06}. 

For case (2) we know that $(M,J)$ admits a (local product) CSC K\"ahler metric in each K\"ahler class, while for case (3) we do not have any smooth CSC K\"ahler metrics but we do have non-CSC smooth extremal K\"ahler metrics in some of the K\"ahler classes as we shall explain in Section \ref{smoothconstruction}.

\subsection{Hamiltonian Circle Actions}\label{hamvf}

Consider the symplectic manifolds $(\grS_g\times S^2,\gro_{k_1,k_2})$ and $(\grS_g\tilde{\times} S^2,\gro_{k_1,k_2})$ as ruled surfaces whose symplectic forms represent the class $\gra_{k_1,k_2}$ of Equation (\ref{KahclassA}), that is $[\gro_{k_1,k_2}]=\gra_{k_1,k_2}$.  Hamiltonian circle actions on ruled surfaces have been treated in \cite{McD88,Aud90,AhHa91,Aud04} and are essentially the same as those of \cite{BoTo11}. The corresponding vector fields are not only Hamiltonian, but are also holomorphic. Thus, they leave the K\"ahler structure invariant. They are typically referred to as {\it Hamiltonian Killing vector fields}, and their corresponding group action by a {\it Hamiltonain Killing circle action}. However, we often shorten this nomenclature to {\it Hamiltonian circle action}. It follows from Lemma 1 of \cite{ACGT11} that a necessary condition for a ruled surface to admit a Hamiltonian circle action is that the $\bbp(E)$ must be decomposable of the form $\bbp(\calo\oplus L_{n})$ where $L_n$ is a complex line bundle of degree $n$. Thus, as discussed in Section \ref{complexsubsec} we obtain the pseudo-Hirzebruch surfaces $S_n$ which is diffeomorphic to $\grS_g\times S^2$ for $n=2m$ even and to $\grS_g\tilde{\times} S^2$ for $n=2m+1$ odd. Since $S_n \cong S_{-n}$, we may assume that $n \geq 0$. Then
$\alpha_{1,k}=[\gro_{1,k}]$ is a K\"ahler class as long as $m<k$. Writing the projective bundle $\pi:\bbp(\calo\oplus L_{n})\ra{1.6} \grS_g$ as  $(w,[u,v])$ where $[u,v]$ are homogeneous coordinates in the $\bbc\bbp^1$ fiber $\bbp(\calo\oplus L_{n}(w))$ at $w\in \grS_g$, the circle action on $\bbp(\calo\oplus L_{n})$ is defined by $\tilde{\cala}_{n}(\grl):\bbp(\calo\oplus L_{n})\ra{1.6} \bbp(\calo\oplus L_{n})$ by $\tilde{\cala}_{n}(\grl)(w,[u,v])=(w,[u,\grl v])$ where $\grl\in \bbc$ with $|\grl|=1$ is holomorphic. These circle actions have two fixed point sets on $\bbp(\calo\oplus L_{n})$, namely, the divisors (zero section) $E_{n}:=\bbp({\mathcal O} \oplus 0)$ and (infinity section)
$E'_{n}:=\bbp(0 \oplus L_{n})$.

\subsection{Conjugacy Classes of Maximal Tori}
In this section we apply the work of Bu{\c{s}}e \cite{Bus10} on equivariant Gromov-Witten invariants to show that our Hamiltonian circle actions are non-conjugate in the group of Hamiltonian isotopies of the symplectic manifold $(\grS_g\times S^2,\gro_{1,k})$. 
Since $\grS_g\times S^2$ is not toric, the circle action $\tilde{\cala}_{2m}(\grl)$ corresponds to a maximal torus in the group $\gH\ga\gm(\grS_g\times S^2,\gro_{1,k})$ of Hamiltonian isotopies, and since we are interested in K\"ahlerian structurs we consider only $m<k$, that is, for each $m=0,\cdots,(k-1)$ we consider the circle subgroups $\tilde{\cala}_{2m}(\grl)\subset \gH\ga\gm(\grS_g\times S^2,\gro_{1,k})$. We shall prove

\begin{theorem}\label{conjmaxtori}
There are exactly $k$ conjugacy classes of maximal tori in $\gH\ga\gm(\grS_g\times S^2,\gro_{1,k})$ represented by the $k$ circle subgroups $\tilde{\cala}_{2m}(\grl)$.
\end{theorem}

\begin{proof}
The proof of this theorem uses equivariant Gromov-Witten (EGW) invariants as described by Bu{\c{s}}e \cite{Bus10} which in turn follows \cite{Giv96,LiTi98,Rua99,LeOn08}. We give only a very brief sketch here and refer to these references for details. The point is that these invariants only depend on the conjugacy class of the circle subgroup. The EGW invariants are obtained as a limit of the so-called parametric Gromov-Witten invariants. Given a symplectic manifold $(N,\gro)$ with a Hamiltonian circle action $\cala_{2m}$, consider the Borel construction $N_{\cala_{2m}}=N\times_{\cala_{2m}}ES^1$ where $S^1\ra{1.6}ES^1\ra{1.6}BS^1$ is the usual universal $S^1$-bundle. Now both $ES^1$ and $BS^1$ are direct limits 
$$ES^1=S^\infty=\lim_{r\rightarrow\infty}S^{2r+1}, \qquad BS^1=\bbc\bbp^\infty= \lim_{r\rightarrow\infty}\bbc\bbp^r,$$ respectively. Thus, the fibration $N\ra{1.5} N_{\cala_{2m}}\ra{1.5} \bbc\bbp^\infty$ is the limit of fibrations $N\ra{1.5} N_{\cala_{2m}}^r\ra{1.5} \bbc\bbp^r$ where $N^r_{\cala_{2m}}=N\times_{\cala_{2m}}S^{2r+1}$. Now in our case the fibers of each bundle $N_{\cala_m}^r$ comes equipped with an induced symplectic form $\gro_{1,k}$ together with a compatible complex structure $J_m$ such that $\cala_m$ is a holomorphic circle action with respect to $J_m$. Moreover, $\gro_{1,k}$ and $J_m$ can be extended to a closed 2-form and endomorphism field on each $N^r_{\cala_m}$. 
Of course, in our case $N=\grS_g\times S^2$ with its symplectic form $\gro_{1,k}$, so we denote this by $N^g_k$ and the total space of the bundles described above by $N^{g,r}_{k,\cala_{2m}}$. Letting $A,F$ denote the homology classes $\grS_g\times \{pt\}, \{pt\}\times S^2$, respectively Bu{\c{s}}e shows that the equivariant Gromov-Witten invariant (with no marked points) 
$$EGW(N^g_k,A-mF):H^*(\overline{\calm}_{g,0},\bbq)\ra{1.8} H^*(BS^1,\bbq)$$
can be written in terms of the ``parametric'' Gromov-Witten invariants on each of pieces $N^{g,r}_{k,\cala_{2m}}$ as
$$EGW(N^g_k,A-mF)=\bigoplus_{r=1}^\infty EGW(N^{g,r}_{k,\cala_{2m}},A-mF)(\grb)u^r.$$
Here $\overline{\calm}_{g,0}$ denotes the Deligne-Mumford compactification of the moduli space of genus $g$ curves with no marked points, and we can take $u$ to be a generator of $H^*(BS^1,\bbz)\subset H^*(BS^1,\bbq)$.
Moreover, Bu{\c{s}}e shows that $EGW(N^{g,r}_{k,\cala_{2m}},A-mF)(\grb)=\pm 1$ if $r=2m+g-1$ and zero otherwise.  Hence, one obtains
\begin{equation}\label{EGW}
EGW(N^g_k,A-mF)=\pm u^{2m+g-1}.
\end{equation}
These invariants are invariant under symplectomorphisms, in fact, they are invariant under deformations of the symplectic form \cite{Rua99}. It follows that $EGW(N^g_k,A-mF)$ depends only on the conjugacy class of the circle action $\cala_{2m}$, and that the action $\cala_{2m'}$ is conjugate to $\cala_{2m}$ under the group $\gH\ga\gm(N^g_k)$ if and only if $m=m'$. 
To show that there are no other conjugacy classes we refer to the last paragraph of the proof of Theorem 6.2 in \cite{BoTo11}. Theorem \ref{conjmaxtori} is now proved.
\end{proof}

\begin{remark}\label{EGrem}
The proof of Theorem \ref{conjmaxtori} given here works for all $g$ including $g=1$, but is more involved than the proof given in \cite{BoTo11}. However, the latter proof which uses an equation in rational homotopy given in \cite{Bus10} doesn't work for the full range $m$ when $g$ is large. See Lemma 4.3 in \cite{Bus10}.  
\end{remark}

\section{Sasakian Geometry on $S^3$-bundles over $\grS_g$}\label{s4}

As the case for $S^2$-bundles over Riemann surfaces, there are exactly two $S^3$-bundles over Riemann surfaces and they are distinguished by their second Stiefel-Whitney class $w_2$. In \cite{BoTo11} we proved this using a very recent result of Kreck and L\"uck \cite{KrLu09}. However, it has been pointed out to us by several people that this is ``well known''. Nevertheless, we could not find a clear statement or proof of this fact in the literature. We did notice that one can adapt the second proof of Lemma 6.9 of \cite{McDSa} to the case of $S^3$-bundles by using Hatcher's proof \cite{Hat83} of the Smale conjecture that ${\rm Diff}(S^3)$ deformation retracts onto $O(4)$ to prove:
\begin{proposition}\label{s3sigma}
Let $\grS_g$ be a Riemann surface of genus $g$. There are precisely two oriented $S^3$-bundles over $\grS_g$, the trivial bundle $\grS_g\times S^3$ with $w_2=0$, and the non-trivial bundle, denoted $\grS_g\tilde{\times} S^3$, with $w_2\neq 0$.
\end{proposition}

It is also well known that there are no exotic differential structures in dimension five, that is, any smooth 5-manifolds that are homeomorphic are diffeomorphic. So there are precisely two diffeomorphism types of $S^3$-bundles over $\grS_g$.

\subsection{Circle Bundles over Riemann Surfaces}
The Sasakian geometry of circle bundles over Riemann surfaces has been studied by Geiges \cite{Gei97} and Belgun \cite{Bel01} (see also Chapter 10 of \cite{BG05}). Indeed, when the genus $g\geq 1$ each deformation class has a constant scalar curvature Sasakian metric. (Actually it has constant $\Phi$-sectional curvature). Furthermore, up to a finite cover the Sasakian structure is regular. Alternatively, an orbifold structure on the base is developable (cf. \cite{BG05} page 107). Thus, we let
$M^3_g$ denote the total space of an $S^1$ bundle over a Riemann surface $\grS_g$ of genus $g\geq 1$ and for simplicity we assume that this bundle arises from a generator in $H^2(\grS_g,\bbz)$. 

There are many inequivalent Sasakian structures on $M^3_g$ with constant scalar curvature. These correspond to the inequivalent K\"ahler structures on the base $\grS_g$ arising from the moduli space $\calm_g$ of complex structures on $\grS_g$. When writing $M^3_g$ we often assume that a transverse complex structure has been chosen without specifying which one. Thus, we write the Sasakian structure with constant scalar curvature on $M^3_g$ as $\cals_1=(\xi_1,\eta_1,\Phi_1,g_1)$ and call it the {\it standard Sasakian structure}. However, when we do wish to specify the complex structure on $\grS_g$ we shall denote it by $\grt\in \gM_g$ and denote the induced endomorphism on the circle bundle by $\Phi_\grt$.

We denote the fundamental group of $M^3_g$ by $\grG_3(g)$. Then from the long exact homotopy sequence of the bundle $S^1\ra{1.5}M^3_g\ra{1.5}\grS_g$ and the fact that $\pi_2(\grS_g)=0$ we have 
\begin{equation}\label{m3hyper}
0\ra{1.8}\bbz\ra{1.8}\grG_3(g)\ra{1.8} \grG_2(g)\ra{1.8} 1
\end{equation}
where $\grG_2(g)$ is the fundamental group of $\grS_g$. So $\grG_3(g)$ is an extension of $\grG_2(g)$ by  $\bbz$. Furthermore, it does not split \cite{Sco83}.

\subsection{The Join Construction}\label{joinsec}
We use the join construction of \cite{BGO06} to describe the diffeomorphism type of our 5-manifolds. 
We describe the join of $M^3_g$ with the weighted 3-sphere $S^3_\bfw$. As mentioned in \cite{BoTo11} this construction involves strict contact structures, that is it only really depends on the contact 1-forms and not on the transverse complex structures which we are free to choose. Recall the weighted sphere as presented in Example 7.1.12 of \cite{BG05}. Let $\eta_0$ denote the standard contact form on $S^3$. It is the restriction to $S^3$ of 1-form $\sum_{i=1}^2(y_idx_i-x_idy_i)$ in $\bbr^4$. Let $\bfw=(w_1,w_2)$ be a weight vector with $w_i\in\bbz^+$. Then the weighted contact form is defined by
\begin{equation}\label{wcon1}
\eta_\bfw =\frac{\eta_0}{\eta_0(\xi_\bfw)}
\end{equation}
with Reeb vector field $\xi_\bfw=\sum_{i=1}^2w_iH_i$ where $H_i$ is the vector field on $S^3$ induced by $y_i\partial_{x_i}-x_i\partial_{y_i}$ on $\bbr^4$.

We denote this weighted sphere by $S^3_{\bfw}$ and consider the manifold $M^3_g\times S^3_{\bfw}$ with contact forms $\eta_1,\eta_\bfw$ on each factor, respectively. There is a 3-dimensional torus $T^3$ acting on $M^3_g\times S^3_{\bfw}$ generated by the Lie algebra $\gt_3$ of vector fields $\xi_1,H_1,H_2$ that leaves both 1-forms $\eta_1,\eta_\bfw$ invariant. Now the join construction \cite{BGO06,BG05} provides us with a new contact manifold by quotienting $M^3_g\times S^3_{\bfw}$ with an appropriate circle subgroup of $T^3$. Let $(x,u)\in M^3_g$ with $x\in \grS_g$ and $u$ in the fiber, and $(z_1,z_2)\in \bbc^2$ with $|z_1|^2+|z_2|^2=1$ so it represents a point on $S^3$. Consider the circle action on $M^3_g\times S^3_{\bfw}$ given by 
\begin{equation}\label{joinact}
(x,u;z_1,z_2)\mapsto (x,e^{il_2\theta}u; e^{-iw_1\theta}z_1,e^{-iw_2\theta}z_2)
\end{equation}
where the action $u\mapsto e^{il_2\theta}u$ is that generated by $l_2\xi_1$.  We also assume, without loss of generality, that $\gcd(l_2,w_1,w_2)=1$. The action (\ref{joinact}) is generated by the vector field $l_2\xi_1-\xi_\bfw$. It has period $1/l_2$ on the $M^3_g$ part, and if $l_1=\gcd(w_1,w_2)$ it will have period $-1/l_1$ on the $S^3_\bfw$ part.  With this in mind, when considering quotients we shall always take the pair $(w_1,w_2)$ to be relatively prime positive integers in which case the infinitesimal generator of the action is given by the vector field $l_2\xi_1-l_1\xi_\bfw$. In order to construct the appropriate contact structure with 1-form  $l_1\eta_1+l_2\eta_\bfw$, we renormalize the vector field and consider
\begin{equation}\label{s1action}
L_\bfw=\frac{1}{2l_1}\xi_1-\frac{1}{2l_2}\xi_\bfw= \frac{1}{2l_1}\xi_1-\frac{1}{2l_2}(w_1H_1+w_2H_2).
\end{equation}
This generates a free circle action on $M^3_g\times S^3_{\bfw}$ which we denote by $S^1(l_1,l_2,\bfw)$. 
\begin{definition}\label{join}
The quotient space of $M^3_g\times S^3_{\bfw}$ by the action $S^1(l_1,l_2,\bfw)$ is called the $(l_1,l_2)$-join of $M^3_g$ and $S^3_\bfw$, and is denoted by $M^3_g\star_{l_1,l_2}S^3_{\bfw}$.
\end{definition}

$M^3_g\star_{l_1,l_2}S^3_{\bfw}$ will be a smooth manifold if $\gcd(l_2,\upsilon_2l_1)=1$ where $\upsilon_2=w_1w_2$. 
Moreover, since the 1-form $l_1\eta_1+l_2\eta_\bfw$ on $M^3_g\times S^3_{\bfw}$ is invariant under $T^3$, we get a contact form on $M^3_g\star_{l_1,l_2}S^3_{\bfw}$, denoted by $\eta_{l_1,l_2,\bfw}$, which is invariant under the factor group $T^2(l_1,l_2,\bfw)= T^3/S^1(l_1,l_2,\bfw)$. The corresponding contact structure is $\cald_{l_1,l_2,\bfw}=\ker\eta_{l_1,l_2,\bfw}$, and the Reeb vector field $R_{l_1,l_2,\bfw}$ of $\eta_{l_1,l_2,\bfw}$ is the restriction to $M^3_g\star_{l_1,l_2}S^3_{\bfw}$ of the vector field
\begin{equation}\label{Reebw}
\tilde{R}_{l_1,l_2,\bfw}=\frac{1}{2l_1}\xi_1+\frac{1}{2l_2}\xi_\bfw
\end{equation}
on $M^3_g\times S^3_{\bfw}$. When working with $R_{l_1,l_2,\bfw}$ we often view this as $\tilde{R}_{l_1,l_2,\bfw}$ modulo the ideal $\cali_L$ generated by $L_\bfw$ in which case we have
\begin{equation}\label{Reebw2}
\tilde{R}_{l_1,l_2,\bfw}=\frac{1}{l_2}(w_1H_1+w_2H_2) \mod \cali_L
\end{equation}
and we identify $R_{l_1,l_2,\bfw}$ with the right hand side.

The quotient of $M^3_g\times S^3_{\bfw}$ by the 2-torus generated by $L_\bfw$ and $R_\bfw$ splits giving the complex orbifold $\grS_g\times \bbc\bbp(\bfw)$ with the product complex structure and symplectic form $\gro=l_1\gro_g+l_2\gro_\bfw$ where $\gro_g,\gro_\bfw$ are the standard symplectic form on $\grS_g$ and $\bbc\bbp(\bfw)$, respectively (see footnote below). Then $M^3_g\star_{l_1,l_2}S^3_{\bfw}$ is the total space of the $S^1$ orbibundle $\pi:M^3_g\star_{l_1,l_2}S^3_{\bfw}\ra{1.6} \grS_g\times \bbc\bbp(\bfw)$ which satisfies $\pi^*\gro =d\eta_{l_1,l_2,\bfw}$. This is the orbifold Boothby-Wang construction, and as shown in \cite{BG00a} the orbifold $M^3_g\star_{l_1,l_2}S^3_{\bfw}$ not only inherits a quasi-regular contact structure, but also a natural Sasakian structure $\cals_\bfw=(\xi_\bfw,\eta_{l,\bfw},\Phi_\bfw,g_\bfw)$ from the product K\"ahler structure on the base. In particular, the underlying CR structure which is inherited from the product complex structure on $\grS_g\times \bbc\bbp(\bfw)$ is $(\cald_{l,\bfw},J_\bfw)$ where $J_\bfw=\Phi_\bfw |_{\cald_{l,\bfw}}$. 

As will be indicated below it is quite difficult to determine the exact diffeomorphism type when $l_2>1$. Indeed, in \cite{BoTo11} it was shown that in the genus one case $M^3_1\star_{l_1,l_2}S^3_{\bfw}$ has non-Abelian fundamental group when $l_2>1$ and is a non-trivial lens space bundle over $T^2$. Moreover, its homotopy type appears also to depend on $l_1$. For this reason we focus our attention here on the case $l_2=1$ where the diffeomorphism type can be determined, and in this case $M^3_g\star_{l_1,l_2}S^3_{\bfw}$ is a smooth 5-manifold. Generally, it follows from Proposition 7.6.7 of \cite{BG05} that $M^3_g\star_{l_1,l_2}S^3_{\bfw}$ is a lens space bundle over $\grS_g$. In particular, for $l_2=1$ we have an $S^3$-bundle over $\grS_g$. For ease of notation we define $M^5_{g,l,\bfw}=M^3_g\star_{l,1}S^3_{\bfw}$ with the contact structure $\cald_{l,\bfw}$ and contact form $\eta_{l,\bfw}=l\eta_1+\eta_\bfw$ with Reeb vector field $R_{l,\bfw}=R_{l_1,1,\bfw}$. Note that  $R_{l,\bfw}=L_\bfw+\xi_\bfw$. So mod the ideal generated by $L_\bfw$, $R_{l,\bfw}$ equals $\xi_\bfw=w_1H_1+w_2H_2$ which is independent of $l$. So the Reeb vector field on  $M^5_{g,l,\bfw}$ is simply $\xi_\bfw$, and  the quotient by its circle action is the product K\"ahler orbifold $\grS_g\times \bbc\bbp(\bfw)$ with K\"ahler form $\gro_l=l\gro_g+\gro_\bfw$ where $\gro_g$ and $\gro_\bfw$ are the standard K\"ahler forms\footnote{By the {\it standard K\"ahler form} on $\bbc\bbp(\bfw)$ we mean the Bochner-flat extremal K\"ahler form described, for example, in \cite{Gau09}.} on $\grS_g$ and $\bbc\bbp(\bfw)$, respectively.

\begin{remark}\label{l2}
At this stage what we know about the Sasakian orbifolds $M^3_g\star_{l_1,l_2}S^3_{\bfw}$ when $l_2>1$ is that they are smooth manifolds if $\gcd(l_2,w_1w_2)=1$ whose fundamental group is a $\bbz_{l_2}$ extension of $\grG_2(g)$. We shall say little more about them in this paper.
\end{remark}

\subsection{The Contact Manifolds $M^5_{g,l,\bfw}$}
Here by convention by diffeomorphism (homeomorphism) type we mean oriented diffeomorphism (homeomorphism) type. By Proposition 7.6.7 of \cite{BG05} we know that $M^5_{g,l,\bfw}$ is an $S^3$-bundle over $\grS_g$. So by Proposition \ref{s3sigma} there are precisely two which are determined by the second Stiefel-Whitney class $w_2(M^5_{g,l,\bfw})\in H^2(M^5_{g,l,\bfw},\bbz_2)$. Furthermore, from the homotopy exact sequence of the $S^3$-bundle we obtain
\begin{equation}\label{joinhomotopy}
\pi_1(M^5_{g,l,\bfw})\approx \pi_1(\grS_g)\approx \grG_2(g),\qquad \pi_2(M^5_{g,l,\bfw}) =0.
\end{equation}

Now $w_2(M^5_{g,l,\bfw})$ is the mod $2$ reduction of the first Chern class of the contact bundle $\cald$. So we begin by  determining $c_1(\cald_{l,\bfw})$. 

\begin{lemma}\label{c1}
Let $\cald_{l,\bfw}$ be the contact structure on $M^5_{g,l,\bfw}$. Then
\begin{equation}\label{c1D}
c_1(\cald_{l,\bfw})=(2-2g-l|\bfw|)\grg
\end{equation}
where $\grg\in H^2(M^5_{g,l,\bfw},\bbz)\approx \bbz$ is a generator and $|\bfw|=w_1+w_2$. Thus, $w_2(M^5_{g,l,\bfw})\equiv l|\bfw|\mod 2$.
\end{lemma}

\begin{proof}
The orbifold canonical divisor of $\grS_g\times \bbc\bbp(\bfw)$ is 
\begin{eqnarray}\label{Korbw0} \notag
K^{orb} &=&K_{\grS_g\times \bbc\bbp^1}+(1-\frac{1}{w_1})E_0+(1-\frac{1}{w_2})E_0 \\ \notag
             &=&-(2-2g)F-2E_0+(1-\frac{1}{w_1})E_0-(1+\frac{1}{w_2})E_0 \\
             &=&-(2-2g)F-\frac{|\bfw|}{w_1w_2}E_0.
\end{eqnarray}

While the orbifold first Chern class $c_1^{orb}$ of $-K^{orb}$ is a rational class in $H^2(\grS_g\times \bbc\bbp(\bfw),\bbq)$, it defines an integral class in the orbifold cohomology $H^2_{orb}(\grS_g\times \bbc\bbp(\bfw),\bbz)$ defined as the cohomology of the classifying space of the orbifold (see Section 4.3 of \cite{BG05}). Namely,
the orbifold first Chern class $c_1^{orb}$ of $\grS_g\times \bbc\bbp(\bfw)$ satisfies 
\begin{equation}\label{c1orb}
p^*c_1^{orb}(\grS_g\times \bbc\bbp(\bfw))=2(1-g)\gra+|\bfw|\grb
\end{equation}
where $p$ is the classifying map of the orbifold $\grS_g\times \bbc\bbp(\bfw)$  and $\gra,\grb$ are the classes in $H^2_{orb}(\grS_g\times \bbc\bbp(\bfw),\bbz)$ representing $\gro_g$ and $\gro_\bfw$, respectively. In fact, $\gra$ is a generator in $H^2(\grS_g,\bbz)$ and $\grb$ is a generator in $H^2_{orb}(\bbc\bbp(\bfw),\bbz)$. It follows from the definition of the $(l,1)$-join that $\gra$ pulls back to a generator and $\grb$ pulls back to $l$ times a generator. Thus, since $\pi^*\gro=d\eta_{l,\bfw}$ we have $l\pi^*\gra+\pi^*\grb=0$. So we can take $\pi^*\gra=\grg$ and $\pi^*\grb=-l\grg$ with $\grg$ a generator, or equivalently $\pi^*PD(E_0)=-lw_1w_2\grg$ where $PD$ denotes Poincar\'e dual. This gives Equation (\ref{c1D}) and proves the result.
\end{proof}

Combining Lemma \ref{c1} and Proposition \ref{s3sigma} we have

\begin{theorem}\label{diffeotype2}
The Sasakian 5-manifold $M^5_{g,l,\bfw}$ is diffeomorphic to $\grS_g\times S^3$ if $l|\bfw|$ is even and diffeomorphic to the non-trivial $S^3$-bundle over $\grS_g$ if $l|\bfw|$ is odd.
\end{theorem}

Theorem \ref{diffeotype2} and Lemma \ref{c1} imply

\begin{corollary}\label{infcont}
There are countably infinite distinct contact structures of Sasaki type on both $\grS_g\times S^3$ and $\grS_g\tilde{\times}~ S^3$.
\end{corollary}

\subsection{The Sasaki Cone and Deformed Sasakian Structures}\label{defss}
Recall the (unreduced) Sasaki cone \cite{BGS06}. Let $\cals_0=(\xi_0,\eta_0,\Phi_0,g_0)$ be a Sasakian structure and let $\gA\gu\gt(\cals_0)$ its group of automorphisms. We denote the Lie algebra of infinitesimal automorphisms of $\cals_0$ by $\ga\gu\gt(\cals_0)$. We let $\gt_k\subset \ga\gu\gt(\cals_0)$ denote the Lie algebra of a maximal torus in $\gA\gu\gt(\cals_0)$, which is unique up to conjugacy. It has rank $k$. The unreduced Sasaki cone is given by
$$\gt_k^+=\{\xi\in \gt_k~|~\eta_0(\xi)>0\}.$$
Here we consider the Sasaki cone of our Sasakian structure $\cals_{l,\bfw}=(\xi_\bfw,\eta_{l,\bfw},\Phi_\bfw,g_\bfw)$. Recall from Section \ref{joinsec} that on $M^3_g\times S^3_\bfw$ we have the Lie algebra $\gt_3$ generated by $\xi_1,H_1,H_2$. Let $\gt_1(\bfw)$ be the Lie algebra generated by the vector field $L_\bfw \in\gt_3$ of Equation (\ref{s1action}) with $l_2=1$ of course. There is an exact sequence of Abelian Lie algebras
$$0\ra{1.8}\gt_1(\bfw)\ra{1.8} \gt_3\ra{1.8} \gt_2(\bfw)\ra{1.8} 0,$$
and we view the quotient algebra $\gt_2(\bfw)=\gt_3/\gt_1(\bfw)$ as a Lie algebra on $M^5_{g,l,\bfw}$.
Then the unreduced Sasaki cone $\gt_2^+(\bfw)$ of $M^5_{g,l,\bfw}$  is defined by
\begin{equation}\label{sascone}
\gt_2^+(\bfw)=\{\xi\in \gt_2(\bfw)~|~\eta_{l,\bfw}(\xi)>0\}.
\end{equation}

We can take $\xi_\bfw,H_1$ as a basis for $\gt_2(\bfw)$. Then for $R\in\gt_2^+(\bfw)$ writing $R=a\xi_\bfw+bH_1$ shows that we must have $a>0$ and $aw_1+b>0$. We can also write 
$$R=a\xi_\bfw+bH_1=(aw_1+b)H_1+aw_2H_2=v_1H_1+v_2H_2=\xi_\bfv$$
which identifies the Sasaki cone of $M^5_{g,l,\bfw}$ with the Sasaki cone of $S^3$. All Sasakian structures $\cals_\bfv=(\xi_\bfv,\eta_{l,\bfv},\Phi_\bfv,g_\bfv)$ in the Sasaki cone $\gt_2^+(\bfw)$ have the same underlying CR structure, namely $(\cald_{l,\bfw},J_\bfw)$. We have
\begin{equation}\label{newsas}
\eta_{l,\bfv}=\frac{\eta_{l,\bfw}}{\eta_{l,\bfw}(\xi_\bfv)}, \qquad \Phi_\bfv |_{\cald_{l,\bfw}}=\Phi_\bfw |_{\cald_{l,\bfw}}=J_\bfw.
\end{equation}
It is the reduced Sasaki cone $\grk(\cald,J)$ that can be thought of as the moduli space of Sasakian structures associated to an underlying CR structure. It is simply the quotient of the unreduced Sasaki cone by the Weyl group of the CR automorphism group. This action amounts to ordering either the $w_i$s or  the $v_i$s. As we shall see shortly it is more convenient to order the $w_i$s.

We first consider the smooth join, that is,  $\bfw=(1,1)$. 

\begin{lemma}\label{sasconeprop}
Consider the Sasakian structure $\cals_{k,(1,1)}=(\xi_{k,(1,1)},\eta_{k,(1,1)},\Phi_\grt,g)$ on the 5-manifold $M^5_{g,k,(1,1)}$ with $\Phi_\grt |_{\cald_{k,(1,1)}}=J \in \gJ$.
Let $X_{2m}$ denote the infinitesimal generator of the induced Hamiltonian circle action on $M^5_{g,k,(1,1)}$. 
\begin{enumerate}
\item If 
$J$ is a case (2) or a case(3)
the Sasaki cone has dimension two and  is determined by 
$$\grk(\cald_{k,(1,1)},J)=\{a\xi_{k,(1,1)}+bX_{2m}~|~a+b\eta_2(X_{2m})>0\},$$
where $\eta_2$ is the standard contact form on $S^3$. 
\item If $J$ is a case (1) the Sasaki cone $\grk(\cald_{k,(1,1)},J)$ has dimension one consisting only of the ray of the Reeb vector field $\xi_{k,(1,1)}$.
\end{enumerate}
\end{lemma}

\begin{proof}
As in \cite{BoTo11} for $i=1,2$ we let $(\eta_i,\xi_i)$ denote the contact 1-form and its Reeb vector field on $M^3_g$ and $S^3$, respectively and consider the commutative diagram
\begin{equation}\label{s1comdia}
\begin{matrix}  M_g^3\times S^3 &&& \\
                          &\searrow && \\
                          \decdnar{} && M^5_{g,k,(1,1)} &\\
                          & \swarrow && \\
                          \grS_g\times S^2 &&&,
\end{matrix}
\end{equation}
where the vertical arrow is the natural $T^2$-bundle projection map generated by the vector fields
\begin{equation}\label{Leqn}
L=\frac{1}{2k}\xi_1-\frac{1}{2}\xi_2,~\qquad \xi_{k}=\frac{1}{2k}\xi_1+\frac{1}{2}\xi_2. 
\end{equation}
The vector field $L$ generates the circle action of the southeast arrow, and $\xi_k$ generates the circle action of the southwest arrow, and it is the Reeb vector field of the contact 1-form $\eta_{k}=k\eta_1+\eta_2$.
\end{proof}

We now want to describe the K\"ahler orbifold associated to the Sasakian structure $\cals_\bfv$ when this structure is quasi-regular. For this purpose we can assume that $v_1$ and $v_2$ are relatively prime positive integers.

\begin{proposition}\label{defsas}
Let $\bfv=(v_1,v_2)$ with $v_1,v_2\in \bbz^+$ and $\gcd(v_1,v_2)=1$, and let $\xi_\bfv$ be a Reeb vector field in the Sasaki cone $\gt_2^+(\bfw)$. Then the quotient of $M^5_{g,l,\bfw}$ by the circle action $S^1(\bfv)$ generated by $\xi_\bfv$ is a complex fiber bundle over $\grS_g$ whose fiber is the complex orbifold $\bbc\bbp(\bfv)$.
\end{proposition}

\begin{proof}
We know from general principles that the quotient $M^5_{g,l,\bfw}/S^1(\bfv)$ is a projective algebraic orbifold with an induced orbifold K\"ahler structure. We denote this K\"ahler orbifold by $B_{\bfv,\bfw}$, and consider the 2-dimensional subalgebra $\gt_2(\bfv,\bfw)$ of $\gt_3$ generated by the vector fields $L_\bfw$ and $\xi_\bfv$ on $M^3_g\times S^3$. The $T^2$ action generated by $\gt_2(\bfv,\bfw)$ on $M^3_g\times S^3_\bfw$ is given by 
\begin{equation}\label{t2action}
(x,u;z_1,z_2)\mapsto (x,e^{i\theta}u;e^{i(v_1\phi-lw_1\theta)}z_1,e^{i(v_2\phi-lw_2\theta)}z_2),
\end{equation}
where $(x,u)\in M^3_g$ with $u$ in the fiber of the bundle $\grr:M^3_g\ra{1.6}\grS_g$, and $(z_1,z_2)\in S^3_\bfw$. By quotienting first by the circle action generated by $L_\bfw$ gives the following commutative diagram
\begin{equation}\label{comdia1}
\begin{matrix}  M^3_g\times S^3_\bfw &&& \\
                          &\searrow && \\
                          \decdnar{\pi_B} && M^5_{g,l,\bfw} &\\
                          &\swarrow && \\
                          B_{\bfv,\bfw} &&& 
\end{matrix}
\end{equation}
where $\pi_B$ is the quotient projection by the 2-torus generated by $\gt_2(\bfv,\bfw)$, the southeast arrow is the quotient projection by the circle action generated by $L_\bfw$, and the southwest arrow is the quotient projection generated by $S^1(\bfv)$. A point of $B_{\bfv,\bfw}$ is given by the equivalence class $[x,u;z_1,z_2]$ defined by the $T^2$ action (\ref{t2action}). We claim that $B_{\bfv,\bfw}$ is a bundle over $\grS_g$ with fiber $\bbc\bbp(\bfv)$. To see this consider the projection $\pi:M^3_g\times S^3_\bfw\ra{1.6} \grS_g$ defined by $\pi=\grr\circ\pi_1$ where $\pi_1:M^3_g\times S^3_\bfw\ra{1.6} M^3_g$ is projection onto the first factor. We have
$$\pi(x,e^{i\theta}u;e^{iv_1\phi-lw_1\theta}z_1,e^{iv_2\phi-lw_2\theta}z_2)= \pi(x,u;z_1,z_2)=x,$$
so the torus acts in the fibers of $\pi$. This gives a map $\tau:B_{\bfv,\bfw}\ra{1.6} \grS_g$ defined by $\tau([x,u;z_1,z_2])=x$, so $\pi$ factors through $B_{\bfv,\bfw}$ giving the commutative diagram
\begin{equation}\label{s2comdia}
\begin{matrix}  M^3_g\times S^3_\bfw &&& \\
                          &\searrow\pi_B && \\
                          \decdnar\pi && B_{\bfv,\bfw} &\\
                          &\swarrow\tau && \\
                          \grS_g &&& .
\end{matrix}
\end{equation}
Furthermore, from the action (\ref{t2action}) the fibers of $\tau$ are the weighted projective spaces $\bbc\bbp(\bfv)$, and the complex structure $\cJ_\bfw$ on $B_{\bfv,\bfw}$ is that induced by $J_\bfw$ on $M^5_{g,l,\bfw}$.
\end{proof}

We call the K\"ahler orbifold $B_{\bfv,\bfw}$ an {\it orbifold pseudo-Hirzebruch surface}. Notice that $B_{\bfv,\bfw}$ inherits a Hamiltonian circle action from the factor algebra $\gt_3/\gt_2(\bfv,\bfw)$. This algebra is generated by the vector field on $B_{\bfv,\bfw}$ induced by say, $H_1$ which by abuse of notation we also denote $H_1$. This vector field is also holomorphic with respect to the complex structure $\cJ_\bfw$ on $B_{\bfv,\bfw}$,

\subsection{Regular Sasakian Structures}\label{regsas}
It follows from Proposition \ref{defsas} that each Sasaki cone $\grk(M^5_{g,l,\bfw},J_\bfw)$ contains a unique ray of regular Sasakian structures determined by setting $\bfv=(1,1)$. Then we have Reeb vector field $R=H_1+H_2$ and $B_{1,\bfw}$ is a pseudo-Hirzebruch surface with trivial orbifold structure. By the Leray-Hirsch Theorem the homology (cohomology) groups are obtained from the tensor product of the homology (cohomology) groups of the base  and the fiber (see Section 1.3 of \cite{ACGT08}). Thus, the first Chern class satisfies 
\begin{equation}\label{c14man}
c_1(B_{1,\bfw})=(2PD(E_n)+(2-2g-n)PD(F)
\end{equation}
where the divisors $E_n$ and $F$ satisfy $E_n\cdot E_n=n, E_n\cdot F=1$ and $F\cdot F=0$. Since the second Stiefel-Whitney class is the mod 2 reduction of $c_1$, we see that $B_{1,\bfw}$ is diffeomorphic to $\grS_g\times S^2$ when $n$ is even and diffeomorphic to $\grS_g\tilde{\times} S^2$ when $n$ is odd.

When $n=2m$ is even, we have $PD(F)=[\gro_g]$ and $PD(E_{2m})=m[\gro_g]+[\gro_0]$  where the class $[\gro_g] ([\gro_0])$ represents the area form of $\grS_g$ (the fiber $\bbc\bbp^1$), respectively.   If $\pi:M^5_{g,l,\bfw}\ra{1.5} B_{1,\bfw}$ denotes the $S^1$ bundle map, we have 
\begin{equation}\label{c1=}
\pi^*c_1(B_{1,\bfw})=c_1(\cald_{l,\bfw}). 
\end{equation}
Writing the symplectic class on $B_{1,\bfw}$ as 
\begin{equation}\label{kahclass1}
[\gro]=k_1[\gro_0]+k_2[\gro_g]=k_1PD(E_{2m}) +(k_2-mk_1)PD(F)
\end{equation}
for some relatively prime positive integers $k_1,k_2$. 

When $n=2m+1$ is odd, we have $PD(F)=[\gro_g]$ and $PD(E_{2m+1})=PD(E_1)+mPD(F)$, so
\begin{equation}\label{kahclass2}
[\gro]=k_1h+k_2[\gro_g]=k_1PD(E_{2m+1}) +(k_2-mk_1)PD(F)
\end{equation}
where $h = PD(E_1)$. 

Let us thus generally write the symplectic class generally as 
\begin{equation}\label{genkahclass}
[\gro]=k_1h+k_2[\gro_g]
\end{equation}
where if $n$ is even $h=[\gro_0]$, and if $n$ is odd $h = PD(E_1)$. 
Then in both cases $h$ and $[\gro_g]$ are primitive integral classes on $B_{1,\bfw}$. 
It is important to realize that the integers $(k_1,k_2,n)$ should be completely determined by the integers $(g,l,w_1,w_2)$.

\begin{lemma}\label{parlem}
The following relations hold:
\begin{enumerate}
\item $n=l|\bfw|-2lw_2=l(w_1-w_2).$
\item $k_1=1$.
\item $k_2 = \begin{cases}
\frac{1}{2}l|\bfw|, &\text{if $l|\bfw|$ is even;}\\
\frac{1}{2}(l|\bfw|-1) &\text{if $l|\bfw|$ is odd.}
\end{cases}$
\end{enumerate} 
Thus, $l$ divides $n$, and the parity of $n$ coincides with the parity of $l|\bfw|$.
\end{lemma}

\begin{proof}
To prove (1) we let $L_n$ denote a line bundle on $\grS_g$ of degree $n$. Then after defining $\chi=\phi-lw_1\theta$ the $T^2$ action (\ref{t2action}) with $\bfv=(1,1)$ becomes 
\begin{equation}\label{reft2action}
(x,u;z_1,z_2)\mapsto (x,e^{i\theta}u;e^{i\chi}z_1,e^{i(\chi+(l|\bfw|-2lw_2)\theta)}z_2),
\end{equation}
So we can identify $B_{1,\bfw}$ with $\bbp(E)$ where $E=\calo\oplus L_n$ where $n$ is given by the equation of the lemma. 

For (2) we notice that the $S^1$-bundle $\grr:M^5_{g,l,\bfw}\ra{1.5} B_{1,\bfw}$ is determined uniquely up to equivariant diffeomorphism by the relatively prime positive integers $(k_1,k_2)$, and we denote its total space by $M^5_{k_2,k_1}$.  When $n$ is even, the base space $B_{1,\bfw}$ is diffeomorphic to $\grS_g\times S^2$, and when $n$ is odd it is diffeomorphic to the non-trivial bundle $\grS_g\tilde{\times} S^2$. The circle bundles over $M^5_{k_2,k_1}$ are classified by elements of $H^2(M^5_{k_2,k_1},\bbz)\approx \bbz$, and up to orientation only one of these has total space $M^3_g\times S^3$. Now $M^3_g\times S^3$ is also the total space of the $T^2$ bundle over $B_{1,\bfw}$ represented by the generators $h,[\gro_g]\in H^2(B_{1,\bfw},\bbz)\approx \bbz^2$. We view this in stages as 
\begin{equation}\label{t2stages}
\begin{matrix}
S^1&\ra{1.8} & M^3_g\times S^3 \\
&& \decdnar{\tau} \\
S^1&\ra{1.8} & M^5_{k_2,k_1} \\
&&\decdnar{\pi} \\
&& B_{1,\bfw}
\end{matrix}
\end{equation}
From the bundle map $\pi$ we have a relation, namely, $\pi^*[\gro]=0$, so we can choose $\pi^*h=-k_2\grg$ and $\pi^*[\gro_g]=k_1\grg$ where $\grg$ is a generator of $H^2(M^5_{k_2,k_1},\bbz)$. So the residual bundle map $\tau$ has $S^1$ action on the fiber of the first factor given by 
$u\mapsto e^{ik_1\theta}u$. Moreover, $\pi_2(M^5_{k_2,k_1})=0$. So the homomorphism $\grd$
in the homotopy exact sequence 
\begin{equation}\label{homexact2}
0\ra{1.8} \bbz\fract{\grd}{\ra{1.8}}\pi_1(M^3_g\times S^3)\ra{1.8} \pi_1(M^5_{k_2,k_1})\ra{1.8} 1
\end{equation}
is multiplication by $k_1$. Thus, we have 
$$\pi_1(M^5_{k_2,k_1})\approx \grG_3(g)/k_1\bbz.$$
But since $M^5_{k_2,k_1}$ is diffeomorphic to $M^5_{g,l,\bfw}$ by \eqref{joinhomotopy} and the exact sequence (\ref{m3hyper}) we have
$$\grG_3(g)/\bbz\approx \grG_2(g)\approx \pi_1(M^5_{g,l,\bfw})\approx \grG_3(g)/k_1\bbz$$
which is a contradiction unless $k_1=1$.

For (3) we notice that the Equations (\ref{c1D}), (\ref{c14man}) and (\ref{c1=}) give the relations
\begin{equation}\label{regconstraint}
(2-2g)(k_1-1)+l|\bfw| = 
\begin{cases}
2k_2, &\text{for $n$ even;}\\
2k_2+k_1 &\text{for $n$ odd}
\end{cases}
\end{equation}
which using (2) proves (3) and finishes the proof.
\end{proof}

Note that for $n\neq 0$ the pseudo-Hirzebruch surfaces $S_n$ and $S_{-n}$ are equivalent, so we can order the relatively prime integers $w_1,w_2$ such that $w_1\geq w_2$.
For notational convenience we write $k_2=k$. Then in both the even and odd cases the conditions for having a K\"ahler class is the same, namely $m<k$. Of course, the underlying 4-manifolds are different in the two cases.

Lemma \ref{parlem} says that for each triple $(l,w_1,w_2)$ of positive integers with $\gcd(w_1,w_2)=1$ determines a pseudo-Hirzebruch surface $S_n$ with a symplectic class given by 
\begin{equation}\label{symclass}
[\gro]=h+k[\gro_g],
\end{equation}
with $k\in\bbz^+$ and where $h=[\gro_0]=PD(E_0)$ in the even case, and $h=PD(E_1)$ in the odd case.
A converse statement also holds, namely

\begin{lemma}\label{evenHir}
With $k\in\bbz^+$ and $n\in\bbn$ fixed, there is a unique positive integer $l$ and unique ordered pair of relatively prime positive integers $(w_1,w_2)$ such that $[\gro]$ of Equation (\ref{symclass}) is a symplectic class on the corresponding $B_{1,\bfw}$, and we have
\begin{equation*}
\bfw=(w_1,w_2)= 
\begin{cases}
(\frac{k+m}{l},\frac{k-m}{l}) &\text{if $n=2m$;} \\
(\frac{k+m+1}{l},\frac{k-m}{l}) &\text{if $n=2m+1$.}
\end{cases}
\end{equation*}
Furthermore, $[\gro]$ is a K\"ahler class in either case if and only if $m<k$.
\end{lemma}

\begin{proof}
Let $k$ and $n$ be given and suppose that $l|\bfw|=l'|\bfw'|$ and $l(w_1-w_2)=l'(w_1'-w'_2)$. Adding and subtracting these equations give $lw_i=l'w'_i$ for $i=1,2$. But since $\gcd(w_1,w_2)=\gcd(w_1',w_2')=1$ this implies $l=l'$ and $\bfw=\bfw'$. The form for $(w_1,w_2)$ follows from Lemma \ref{parlem}, and the fact that $[\gro]$ is a K\"ahler class was discussed above.
\end{proof}

Lemma \ref{parlem} or Lemma \ref{evenHir} easily implies

\begin{corollary}\label{lemcor1}
If $n=0$ we must have $\bfw=(1,1),l=k$ and $B_{1,\bfw}$ is biholomorphic to $\grS_g\times \bbc\bbp^1$ with its product complex structure. If $n=1$ we must have $l=1,\bfw=(k+1,k)$ with $B_{1,\bfw}$ diffeomorphic to $\grS_g\tilde{\times} S^2$.
\end{corollary}

Our construction together with Lemma \ref{c1} also gives
\begin{proposition}\label{regunique}
There are a countably infinite number of inequivalent contact structures $\cald_{k,(1,1)}$ of Sasaki type on $\grS_g\times S^3$ with $c_1(\cald_{k,(1,1)})=(2-2g-2k)\grg$, and a countably infinite number of inequivalent contact structures $\cald_{1,(k+1,k)}$ of Sasaki type on $\grS_g\tilde{\times}S^3$ with $c_1(\cald_{1,(k+1,k)})=(2-2g-2k-1)\grg$. Moreover, for each CR structure $(\cald_{l,\bfw},J_\bfw)$ there is precisely one regular ray in the Sasaki cone $\grk(\cald_{l,\bfw},J_\bfw)$, namely that determined by $\bfv=(1,1)$. 
\end{proposition}

Lemmas \ref{parlem} and \ref{evenHir} allow us to view the CR structure $(\cald_{l,\bfw},J_\bfw)$ as arising in two natural distinct but  equivalent ways. It either arises from a product complex structure involving weighted projective spaces, or it arises from a smooth pseudo-Hirzebruch surface. In the former viewpoint the CR structure is determined by the weights, whereas, in the latter it is determined by the degree of the line bundle $L$. This can be conveniently illustrated by the follow diagram:
\begin{equation}\label{prodHir}
\begin{matrix} && M_{g,l,\bfw} && \\
                        &&&&  \\
                        &  \swarrow  &&\searrow & \\
                       & \grS_g\times \bbc\bbp(\bfw) &&& B_{1,\bfw}=S_n.
\end{matrix}
\end{equation}
Of course, depending on the parity of $l|\bfw|$ the contact 5-manifold $M_{g,l,\bfw}$ is either $\grS_g\times S^3$ or $\grS_g\tilde{\times}S^3$. In the present formulation we began our construction using the southwest arrow, but we could have equally as well begun with the southeast arrow. In fact, this latter approach was the taken in \cite{BoTo11}. We can view the complex structure on either base as determining the transverse complex structure on $M_{g,l,\bfw}$. Note that when $\bfw=(1,1)$ we have $n=0$ and the two projections coincide.

\begin{example}\label{eleven} Here we give examples of a fixed contact structure $\cald_k$ on each of the 5-manifolds in terms of tables. The first table is on the trivial bundle $\grS_g\times S^3$ with contact structure $\cald_4$, so $c_1(\cald_4)=(2-2g-8)\grg.$ The second table is on the nontrivial bundle $\grS_g\tilde{\times} S^3$ again with contact structure $\cald_4$, so $c_1(\cald_4)=(2-2g-9)\grg.$  In both cases we must have $m<4$, so $m=0,1,2,3$. One can think of each $m$ as labeling a Sasaki cone in a 4-bouquet as described in Section \ref{bouqsec} below\footnote{Strictly speaking we have only proven that distinct values of $m$ give distinct conjugacy classes of maximal tori in the case of the trivial bundle $\grS_g\times S^3$. Thus, only in this case do we actually have what we defined as a 4-bouquet.}. Equivalently, $m$ labels the transverse complex structure.

\bigskip
\centerline{$\grS_g\times S^3$ with contact structure $\cald_4$}
\begin{center}
\begin{tabular}{| l || l | l |}
\hline
m & l & \bfw \\ \hline
0 & 4 & (1,1) \\ \hline
1 & 1 & (5,3) \\ \hline
2 & 2 & (3,1) \\ \hline
3 & 1 & (7,1) \\ \hline
\end{tabular}
\end{center}
\medskip

\bigskip
\centerline{$\grS_g\tilde{\times} S^3$ with contact structure $\cald_4$}
\begin{center}
\begin{tabular}{| l || l | l |}
\hline
m & l & \bfw \\ \hline
0 & 1 & (5,4) \\ \hline
1 & 3 & (2,1) \\ \hline
2 & 1 & (7,2) \\ \hline
3 & 1 & (8,1) \\ \hline
\end{tabular}
\end{center}

\end{example}

\subsection{Quasi-regular Sasakian Structures}\label{quasi-reg-ray}
We now consider the general case $\bfv=(v_1,v_2)$ where $v_1,v_2\in\bbz^+$ and we assume that $\gcd(v_1,v_2)=1$. The base space $B_{\bfv,\bfw}$ is now an orbifold pseudo-Hirzebruch surface as discussed in Section \ref{defss}. As a complex manifold $B_{\bfv,\bfw}$ is a smooth pseudo-Hirzebruch surface $S_n$ for some $n\in\bbz$, but there are branch divisors  making the orbifold structure crucial. It is thus convenient to represent $B_{\bfv,\bfw}$ as a log pair $(B_{1,\bfw'},\grD)$ for some weight vector $\bfw'$ and some branch divisor $\grD$. To do this we consider the map $\th_\bfv:M_g^3\times \bbc^2\backslash\{(0,0)\}\ra{1.6} M_g^3\times \bbc^2\backslash\{(0,0)\}$ defined by
\begin{equation}\label{th}
\th(x,u;z_1,z_2)=(x,u;z_1^{v_2},z_2^{v_1}).
\end{equation}
It is a $v_1v_2$-fold covering map. Consider the $S^1\times\bbc^*$ action $\cala_{\bfv,l,\bfw}(\grl,\grt)$ on $M^3_g\times \bbc^2$ defined by
\begin{equation}\label{act1}
\cala_{\bfv,l,\bfw}(\grl,\grt)(x,u;z_1,z_2)=(x,\grl u;\grt^{v_1}\grl^{-lw_1}z_1,\grt^{v_2}\grl^{-lw_2}z_2),
\end{equation}
where $\grl,\grt\in\bbc^*$ with $|\grl|=1$. Almost by definition we have 
$$B_{\bfv,\bfw}=\bigl(M_g^3\times \bbc^2\backslash\{(0,0)\}\bigr)/\cala_{\bfv,l,\bfw}(\grl,\grt).$$
A straightforward computation gives a commutative diagram:
\begin{equation}\label{actcomdia}
\begin{matrix}
M^3_g\times \bbc^2\backslash\{(0,0)\} &\fract{\cala_{\bfv,l,\bfw}(\grl,\grt)}{\ra{2.5}} & M^3_g\times \bbc^2\backslash\{(0,0)\} \\
\decdnar{\th_\bfv} && \decdnar{\th_\bfv} \\
M^3_g\times \bbc^2\backslash\{(0,0)\} & \fract{\cala_{1,l,\bfw'}(\grl,\grt^{v_1v_2})}{\ra{2.5}} & M^3_g\times \bbc^2\backslash\{(0,0)\},
\end{matrix}
\end{equation}
where $\bfw'=(v_2w_1,v_1w_2)$.
Now $\th_\bfv$ induces a fiber preserving biholomorphism $h_\bfv:B_{\bfv,\bfw}\ra{1.7} B_{1,\bfw'}$ given by $h(x,[z_1,z_2])=(x,[z_1^{v_2},z_2^{v_1}])$. As ruled surfaces $B_{1,\bfw'}=S_n$ where $n=l(w_1v_2-w_2v_1)$.
We can thus write $B_{\bfv,\bfw}$ as the log pair $(S_n,\grD_\bfv)$ where $\grD_\bfv$ is the branch divisor 
\begin{equation}\label{grD}
\grD_\bfv=(1-\frac{1}{v_1})E_n+ (1-\frac{1}{v_2})E'_{n}
\end{equation}
where $E'_{n}$ is the infinity section which satisfies $E'_{n}\cdot E'_{n}=-n$. We have arrived at:
\begin{lemma}\label{quasireglem}
The orbifold pseudo-Hirzebruch surface $B_{\bfv,\bfw}$ can be realized as the orbifold log pair $(S_n,\grD_\bfv)$ where $S_n$ is a pseudo-Hirzebruch surface of degree $n=l(w_1v_2-w_2v_1)$ and the branch divisor $\grD_\bfv$ is given by Equation (\ref{grD}).
\end{lemma}

Notice that when $\bfv=(1,1)$ we obtain the regular structures studied in Section \ref{regsas}, whereas, if $\bfv=\bfw$ we get the product structure $(S_0,\grD_\bfw)=\grS_g\times \bbc\bbp[\bfw]$ with which we started the join construction. Conversely, $n=0$ implies $\bfv$ is proportional to $\bfw$.

Note that the identity map $\BOne:B_{\bfv,\bfw}=(S_n,\grD_\bfv)\ra{1.8} (S_n,\emptyset)$ is a Galois covering map with trivial Galois group (cf. \cite{GhKo05}), where $(S_n,\emptyset)$ is just the ordinary pseudo-Hirzebruch surface $S_n$ with no branch divisors and its usual complex manifold atlas. 

First we compute the orbifold first Chern class. The computation is similar to that  for $B_{1,\bfw}$ with the exception that we replace $c_1(B_{1,\bfw})$ with $c_1^{orb}(B_{\bfv,\bfw})$. So we need the orbifold canonical divisor $K^{orb}$ of $B_{\bfv,\bfw}$. It is given by
\begin{equation}\label{KorbB}
K^{orb}=K_{S_n}+(1-\frac{1}{v_1})E_n+(1-\frac{1}{v_2})E_n'
\end{equation}
where $E_n\cdot E_n=n$ and $E_n'\cdot E_n'=-n$. Then, using
$$K_{S_n}=-2E_n-(2-2g-n)F,$$ 
we have
\begin{equation}\label{KorbB2}
-K^{orb}=(2-2g-n)F+(1+\frac{1}{v_1})E_n-(1-\frac{1}{v_2})E'_n.
\end{equation}
Now the divisors $F,E_n,E'_n$ and $K^{orb}$ are Poincar\'e dual to cohomology classes in $H^2(S_n,\bbq)$, but they pullback to integral classes on $M^5_{g,l,\bfw}$. We know from the proof of Lemma \ref{parlem} that $PD(F)$ pulls back to a generator $\grg$, i.e. $\pi^*PD(F)=\grg$, and we have $\pi^*PD(E_n)=-r_1v_1\grg$ and $\pi^*PD(E'_n)=-r_2v_2\grg$ for some integers $r_1,r_2$. Furthermore, the relation  $E'_n=E_n-nF$ implies $v_1r_1-v_2r_2=-n$. We also have $\pi^*c_1(-K^{orb})=c_1(\cald_{l,\bfw})=(2-2g-l|\bfw|)\grg$. Combining this with Equation (\ref{KorbB2}) gives the system
\begin{eqnarray}\label{msys}
v_1r_1-v_2r_2&=&-n \notag \\
r_1+r_2&=& l|\bfw|
\end{eqnarray}
which gives solutions $r_1=lw_2$ and $r_2=lw_1$. We are now ready for
\begin{lemma}\label{qregk}
With the K\"ahler class on $(B_{1,\bfw'},\grD_\bfv)$ given by Equation (\ref{symclass}), we have 
$$k = \begin{cases}
\frac{1}{2}l|\bfw'|, &\text{if $l|\bfw'|$ is even;}\\
\frac{1}{2}(l|\bfw'|-1) &\text{if $l|\bfw'|$ is odd,}
\end{cases}$$
where $\bfw'=(w_1v_2,w_2v_1)$.
\end{lemma}

\begin{proof}
From Equations (\ref{kahclass1}) and (\ref{kahclass2}) with $k_1=1$ and $k_2=k$, as follows from an argument similar to the proof of (2) in Lemma \ref{parlem}, the K\"ahler class $[\gro]=PD(E_n)+(k-m)PD(F)$ and pulls back to $0$. This gives 
$$0=\pi^*PD(E_n)+(k-m)\pi^*PD(F)=-r_1v_1\grg+(k-m)\grg,$$
where $2m=l(w_1v_2-w_2v_1)$ if $n$ is even and $2m=l(w_1v_2-w_2v_1)-1$ if $n$ is odd. Putting $r_1=lw_2$ into this equation gives the result.
\end{proof}

\section{Families of Sasakian Structures}\label{sasfam}

As with K\"ahlerian structures Sasakian structures occur in families. The analogies go much further, but there are notable differences. A K\"ahler structure has an underlying symplectic structure as well as an underlying complex structure. A Sasaki structure has an underlying contact structure as well as an underlying transverse complex structure, or somewhat equivalently an underlying strictly pseudoconvex CR structure. One can think of contact structures as odd dimensional versions of symplectic structures, and thus, Sasakian structures as an odd dimensional version of K\"ahlerian structures. However, unlike the symplectic case fixing a contact structure does not fix the contact 1-form. So a Sasakian structure requires a bit more information, namely, we need to fix a contact 1-form within the underlying contact structure. Even then there is no guarentee that the chosen contact form is associated to a Sasakian structure. In fact in most cases it is not. In order that a contact 1-form $\eta$ be the 1-form of a Sasakian structure it is necessary that its Reeb vector field $\xi$ be an infinitesimal automorphism of the underlying CR structure. This is equivalent to being a Killing vector field with respect to a compatible Riemannian metric $g$.

\subsection{The Underlying Contact Structure}
According to Theorem \ref{diffeotype2} the diffeomorphism type of our 5-manifold $M^5_{g,l,\bfw}$ is determined by the genus $g$ and the parity of $l|\bfw|$. With the diffeotype fixed, Lemma \ref{c1} says that the underlying contact structure is then determined by $l|\bfw|$ itself. However, we easily see from Lemma \ref{parlem} that the first Chern class of the contact bundle satisfies
\begin{equation}\label{c1k}
c_1(\cald)=\begin{cases} 
                  2-2g-2k, &\text{if $M^5_{g,l,\bfw}=\grS_g\times S^3$;} \\
                  2-2g-2k-1 &\text{if $M^5_{g,l,\bfw}=\grS_g\tilde{\times} S^3$.}
                  
   \end{cases}               
\end{equation}
Recall the two distinct, but completely equivalent, approaches to describing the underlying CR structure represented by diagram (\ref{prodHir}). Lemmas \ref{parlem} and \ref{evenHir} imply that instead of $l$ and $\bfw$ we can label our CR manifolds $M^5_{g,l,\bfw}$ by $n$ and $k$ and write $M^5_{g,k,n}$ in which case we see that $c_1(\cald)$ is independent of $n$ except for its parity. Equation (\ref{c1k}) also implies that the contact structures with different $k$ are inequivalent. Thus, we label our contact structures by $k$ and $n$ and write $\cald_{k,n}$ for $k\in \bbz^+$. We understand that for each $k\in \bbz^+$, $\cald_{k,n}$ is a contact structure on each of the 5-manifolds $\grS_g\times S^3$ when $n$ is even and on $\grS_g\tilde{\times} S^3$ when $n$ is odd. We suppress the dependence on $g$ when labeling the contact structure. In fact we have the stronger result:

\begin{proposition}\label{ineqcontact}
The contact structures $\cald_{k,n}$ and $\cald_{k',n'}$ on a fixed manifold $\grS_g\times S^3$ or $\grS_g\tilde{\times} S^3$ are contactomorphic if and only if $k'=k$ and $(-1)^{n'}=(-1)^n$.
\end{proposition}

\begin{proof}
It remains to prove the `if' part. So we assume $k'=k$. Given the two CR structures $(\cald_{k,n'},J_{n'})$ and $(\cald_{k,n},J_n)$ we choose the unique regular Reeb vector guarenteed by Proposition \ref{regunique}. In each case the quotient manifold $B_g$ by the $S^1$-action generated by the Reeb field is $\grS_g\times S^3$ if $n$ is even and $\grS_g\tilde{\times} S^3$ if $n$ is odd. Moreover, in either case the symplectic class in $H^2(B_g,\bbz)$ is given by Equation (\ref{symclass}) which is independent of $n$ and $n'$. So the corresponding  symplectic forms $\gro$ and $\gro'$ are cohomologous. By a theorem of Lalonde and McDuff \cite{LaMcD96} these two forms are symplectomorphic. That is there is a diffeomorphism $\phi:B_g\ra{1.6} B_g$ such that $\phi^*\gro'=\gro$. By the Boothby-Wang construction \cite{BoWa} there are isomorphic principal circle bundles $\pi:M\ra{1.6} B_g$ and $\pi':M\ra{1.6} B_g$ over each symplectic manifold with connection 1-forms $\eta$ and $\eta'$, respectively, such that $d\eta=\pi^*\gro$ and $d\eta'=\pi'^*\gro'$. Moreover, the connection forms are also contact 1-forms and $\phi$ extends to a fiber preserving  diffeomorphism $\tphi:M\ra{1.6} M$ such that $\tphi^*\eta'=\eta$. So the contact structures are contactomorphic. 
\end{proof}

Proposition \ref{ineqcontact} implies that the isomorphism class of contact structure is independent of $n$ up to its parity which determines which of the two $S^3$-bundles occur. We shall henceforth denote our contact structures by $\cald_k$ making clear when necessary which manifold is involved.

\subsection{Families of Sasakian Structures Associated to $\cald_{k}$}\label{famsas}
Families of Sasakian structures arise from deformations in various ways. First we fix a Sasakian structure $\cals=(\xi,\eta,\Phi,g)$. Then one can deform the contact form $\eta$ within the contact structure $\cald$ by sending $\eta\mapsto f\eta$ with $f>0$ everywhere. This fixes the CR structure, $(\cald,J)$, and under certain rather stringent conditions a Sasakian structure will remain Sasakian. This gives rise to Sasaki cones and is called a {\it deformation of type I} in \cite{BG05}. Next, one can deform the contact structure $\cald$ by sending $\eta\mapsto \eta +\grz$ where $\grz$ is basic 1-form. This keeps the characteristic foliation $\calf_\xi$ fixed while deforming the transverse K\"ahler structure $d\eta$ within its basic cohomology class. Gray's Theorem says that any such deformations give equivalent ({\it isotopic}) contact structures.  This is called a {\it deformation of type II} in \cite{BG05}. Finally, one can deform the transverse complex structure by sending $J\mapsto J_t$ where $J=\Phi |_\cald$ while keeping the contact structure $\cald$ fixed. This last type gives rise to the jumping phenomenon known in ruled surfaces \cite{MoKo06,Suw69,Sei92}. It follows from Proposition \ref{regunique} that  all the CR structures described in Section \ref{s4} are induced by a regular Boothby-Wang construction over a ruled surface of type (2) with $L$ trivial and type (3). However, we can choose the vector bundle $E$, hence, the line bundle $L$ to be given by a reducible representation $\grr$ of $\pi_1(\grS_g)$. Using Proposition \ref{hamred} gives all complex structures of type (2)  for the case $n=0$, and from the discussion of Section \ref{s2.3} we similarly obtain all complex structures of type (3). All of these induced CR structures have 2-dimensional Sasaki cones. In addition there are complex structures of case (1) which induce CR structures with a 1-dimensional Sasaki cone. Again by Proposition \ref{hamred} the CR structures induced by case (1)(a) are determined by the complex structure on $\grS_g$ together with an irreducible representation of $\pi_1(\grS_g)$. Then using Lemma \ref{Kahclass}  we have arrived at

\begin{lemma}\label{J}
Consider the regular contact manifolds $(\grS_g\times S^3,\cald_k)$ and $(\grS_g\tilde{\times} S^3,\cald_k)$ where $k$ is a positive integer which are the total spaces of principal $S^1$-bundles over the symplectic manifolds $(\grS_g\times S^2,\gro_{1,k})$ and $(\grS_g\tilde{\times} S^2,\gro_{1,k})$, respectively. Then with the exception of case (1)(b) the induced complex structures $J$ in $(\cald_k,J)$ are determined by the triple $(\grt,\grr,m)$ where $\grt\in\gM_g$, $\grr$ is a representation of $\pi_1(\grS_g)$ which is irreducible for case (1)(a), reducible for case (2), and $m$ is a nonnegative integer.
\end{lemma}

Henceforth, we shall label our transverse complex structures as $J_{\grt,\grr,m}$.

\subsection{Bouquets of Sasaki Cones}\label{bouqsec}
The concept of bouquets of Sasakian structures was introduced in \cite{Boy10a}. They consist of a discrete (usually finite) number of Sasaki cones that are associated with the same isomorphism class of contact structure. These Sasaki cones are associated to (almost) CR structures whose maximal tori in $\gC\gR(\cald,J)$ belong to distinct conjugacy classes of tori in the contactomorphism group $\gC\go\gn(\cald)$. But typically a Sasaki cone may be associated to more than one complex ruled surface, hence, to more than one CR structure with the same underlying contact structure. For example, in the cases treated here, varying the complex structure $\grt$ on $\grS_g$ does not change the Sasaki cone; however, when $m=0$ varying the representation $\grr$ from reducible to irreducible or vice-versa does. 

We wish to describe how  the Sasaki cones vary with the transverse complex structure $J$. We have
\begin{lemma}\label{sasconecomp}
Let $\cals_{\grt,\grr,m}$ be a Sasakian structure on $(\grS_g\times S^3,\cald_k)$ or $(\grS_g\tilde{\times} S^3,\cald_k)$ with underlying CR structure $(\cald_k,J_{\grt,\grr,m})$. Then 
\begin{enumerate}
\item If $m=0$ and $\grr\in \gR(\grS_g)^{irred}$, then $\dim\grk(\cald_k,J_{\grt,\grr,0})=1$ and $\grk(\cald_k,J_{\grt,\grr,0})=\grk(\cald_k,J_{\grt',\grr',0})$ for any $\grt,\grt'\in\gM_g$ and $\grr,\grr'\in \gR(\grS_g)^{irred}$.
\item If $0\leq m<k$ and $\grr\in \gR(\grS_g)^{red}$, then $\dim\grk(\cald_k,J_{\grt,\grr,m})=2$ and $\grk(\cald_k,J_{\grt,\grr,m})=\grk(\cald_k,J_{\grt',\grr',m})$ for any $\grt,\grt'\in\gM_g$ and $\grr,\grr'\in \gR(\grS_g)^{red}$.
\item The Sasaki cone for any complex structure of type (1)(b) coincides with the Sasaki cone for any complex structure of type (1)(a), that is, complex structures of item {\rm (1)} above.
\end{enumerate}
\end{lemma}

\begin{proof}
The result follows in a straightforward manner from our previous results and the Boothby-Wang construction. We leave the details to the reader. 
\end{proof}

\begin{remark}\label{sameReeb}
Fix the 5-manifold to be either $\grS_g\times S^3$ or $\grS_g\tilde{\times}S^3$. Then we notice from the proof of Lemma \ref{sasconecomp} that all the Sasaki cones contain the same Reeb vector field corresponding to the principal $S^1$-bundle over either $\grS_g\times S^2$ or $\grS_g\tilde{\times} S^2$ defined by the cohomology class $\gra_{1,k}=[\gro_{1,k}]$. The differences in the cones come from changing the complex structure on the base. For example, some complex structures admit a Hamiltonian Killing vector field and some don't.
\end{remark}

We now construct our Sasaki bouquets. Recall \cite{Boy10a} the map $\gQ$ that associates to any transverse almost complex structure $J$ that is compatible with the contact structure $\cald$, a conjugacy class of tori in $\gC\go\gn(\cald)$, namely the unique conjugacy class of maximal tori in $\gC\gR(\cald,J)\subset \gC\go\gn(\cald)$. Then two compatible transverse almost complex structures $J,J'$ are {\it T-equivalent} if $\gQ(J)=\gQ(J')$. A {\it Sasaki bouquet} is defined by
\begin{equation}\label{sasbouq}
\gB_{|\cala|}(\cald)=\bigcup_{\gra\in\cala}\grk(\cald,J_\gra)
\end{equation}
where the union is taken over one representative of each $T$-equivalence class in a preassigned subset $\cala$ of $T$-equivalence classes of transverse (almost) complex structures. Here $|\cala|$ denotes the cardenality of $\cala$.

We can now build our Sasaki bouquets. Here we treat only the 5-manifold $\grS_g\times S^3$ since to distinguish $T$-equivalence classes we need to use Theorem \ref{conjmaxtori} and we do not at this stage have this theorem for the non-trivial bundle $\grS_g\tilde{\times} S^3$. We do, however, believe that such a result holds. We consider the indexing set $\cala$ to be the set $\{0,\ldots,k-1\}\sqcup \{0'\}$ where $0'$ is a disjoint copy of zero. 

\begin{theorem}\label{sasbouquetthm}
For $g\geq 2$ the contact manifold $(\grS_g\times S^3,\cald_k)$ admits a $k+1$-bouquet 
$$\gB_{k+1}(\cald_k)=\bigcup_{m=0}^{k-1}\grk(\cald_k,J_{\grt,\grr,m})\bigcup \grk(\cald_k,J_{\grt',\grr',0'}),$$
where $\grt,\grt'$ are any complex structures on $\grS_g$, $\grr$ is any reducible representation of $\pi_1(\grS_g)$ and $\grr'$ is any irreducible representation of $\pi_1(\grS_g)$. Moreover, all of the Sasaki cones in $\gB_{k+1}(\cald_k)$ intersect in the ray of the regular Reeb vector field of the Boothby-Wang fibration 
$$(\grS_g\times S^3,\eta_k)\ra{1.6} (\grS_g\times S^2,\gro_{1,k}).$$
\end{theorem}

\begin{proof}
By Lemma \ref{sasconecomp} we know that for $m=0,\ldots,k-1$ the dimension of the Sasaki cones $\grk(\cald_k,J_{\grt,\grr,m})$ is two when $\grr$ is reducible, and the dimension of $\grk(\cald_k,J_{\grt',\grr',0})$ is one when $\grr'$ is irreducible. Furthermore, from the Boothby-Wang construction they intersect in the ray of the Reeb vector field of $\eta_k$. Now suppose that $0\leq m'< m<k$, and let $\tilde{\cala}_m$ and $\tilde{\cala}_{m'}$ denote the corresponding Hamiltonian circle actions as described in the beginning of Section \ref{hamvf}. Then by Theorem \ref{conjmaxtori} these correspond to non-conjugate maximal tori in $\gH\ga\gm(\grS_g\times S^2,\gro_{1,k})$. By \cite{Ler02b,Boy10a} these Hamiltonian circle groups lift to $\grS_g\times S^3$ giving maximal tori of dimension two in the contactomorphism group $\gC\go\gn(M^5_{g,k},\eta_{k})$ that are non-conjugate in the larger group $\gC\go\gn(M^5_{g,k},\cald_{k})$. It follows that $\gQ(J_m)\neq \gQ(J_{m'})$ giving the first union in the bouquet. 

It is well known (cf. \cite{New02}) that when $g\geq 2$ the moduli space of stable vector bundles of rank 2 and degree 0 over $\grS_g$ is non-empty, and is realized by irreducible unitary representations of $\pi_1(\grS_g)$ \cite{NaSe65}. This gives the 1-dimensional Sasaki cone of the bouquet. Moreover, we clearly have $\gQ(J_m)\neq \gQ(J_{0'})$ since the 
latter has a 1-dimensional Sasaki cone and the former a 2-dimensional Sasaki cone. 
\end{proof}

\begin{remark}\label{A_0rem}
Of course, we can construct a $k+1$-bouquet in the $g=1$ case by adding the Sasakian structure induced by the complex structure $A_0$ (see \cite{Suw69,BoTo11} for the definition of the notation) to the $k$ 2-dimensional Sasaki cones. However, this complex structure is unstable and is not related to unitary representations of $\pi_1(\grS_g)$.
\end{remark}

We also have the following two results:

\begin{proposition}\label{contconj}
The contactomorphism group $\gC\go\gn(\grS_g\times S^3,\cald_{k})$ contains at least $k$ conjugacy classes of maximal tori of dimension $2$ of Reeb type, and exactly $k$ conjugacy classes of maximal tori of dimension $2$ of Reeb type that intersect in the ray of the Reeb vector field $\xi_{k}$.
\end{proposition}

\begin{proof}
This is an immediate consequence of Theorem \ref{conjmaxtori} together with the lifting results in \cite{Ler02b,Boy10a}.
\end{proof}

\begin{theorem}\label{sasbouq}
The manifold $\Sigma_g\times S^3$ admits a countably infinite number of distinct contact structures $\cald_k$ labelled by $k\in\bbz^+$ each having a Sasaki $k$-bouquet of Sasakian structures consisting of 2-dimensional Sasaki cones intersecting in a ray.
\end{theorem}

\begin{proof}
The fact that the contact structures $\cald_{k}$ and $\cald_{k'}$ are inequivalent when $k'\neq k$ is Proposition \ref{ineqcontact}. The statement about the bouquets is a consequence of the discussion above and Theorem \ref{conjmaxtori}.
\end{proof}

\subsection{Extremal Sasakian Structures}\label{ess}
In this section we will discuss how to lift our extremal K\"ahler metrics to extremal Sasaki metrics via the Boothby-Wang construction. When appropriate, we then deform in the Sasaki cone to obtain quasiregular Sasakian structures which project to K\"ahler orbifolds which in some (but not all) cases have so-called {\em admissible} extremal representatives. As in \cite{Boy10a,Boy10b,BoTo11} this will give rise to extremal Sasakian structures. Due to the higher genus of $\Sigma_g$, in case (3), the existence of {\em admissible} extremal K\"ahler metrics and the corresponding extremal Sasakian metrics is not universally given. 

To make this paper self-contained we shall first remind the reader of the definition of Extremal Sasakian metrics.
Given a Sasakian structure $\cals=(\xi,\eta,\Phi,g)$ on a compact manifold $M^{2n+1}$ we deform the contact 1-form by $\eta\mapsto \eta(t)=\eta+t\grz$ where $\grz$ is a basic 1-form with respect to the characteristic foliation $\calf_\xi$ defined by the Reeb vector field $\xi.$ Here $t$ lies in a suitable interval containing $0$ and such that $\eta(t)\wedge d\eta(t)\neq 0$. This gives rise to a family of Sasakian structures $\cals(t)=(\xi,\eta(t),\Phi(t),g(t))$ that we denote by ${\mathcal S}(\xi, \bar{J})$ where $\bar{J}$ is the induced complex structure on the normal bundle $\nu(\calf_\xi)=TM/L_\xi$ to the Reeb foliation $\calf_\xi$ which satisfy the initial condition $\cals(0)=\cals$. On the space ${\mathcal S}(\xi, \bar{J})$  we consider the ``energy functional'' $E:{\mathcal S}(\xi,\bar{J})\ra{1.4} \bbr$ defined by
\begin{equation}\label{var}
E(g) ={\displaystyle \int _M s_g ^2 d{\mu}_g ,}\, 
\end{equation}
i.e. the $L^2$-norm of the scalar curvature $s_g$ of the Sasaki metric $g$. Critical points $g$ of this functional are called {\it extremal Sasakian metrics}.  Similar to the K\"ahlerian case, the Euler-Lagrange equations for this functional says \cite{BGS06} that $g$ is critical if and only if the gradient vector field $J{\rm grad}_gs_g$ is transversely holomorphic, so, in particular, Sasakian metrics with constant scalar curvature are extremal. Since the scalar curvature $s_g$ is related to the transverse scalar curvature $s^T_g$ of the transverse K\"ahler metric by $s_g=s_g^T-2n$, a Sasaki metric is extremal if and only if its transverse K\"ahler metric is extremal. Hence, in the regular (quasi-regular) case, an extremal K\"ahler metric lifts to an extremal Sasaki metric, and conversely an extremal Sasaki metric projects to an extremal K\"ahler metric. Note that the deformation $\eta\mapsto \eta(t)=\eta+t\grz$ not only deforms the contact form, but also deforms the contact structure $\cald$ to an equivalent (isotopic) contact structure. So when we say that the contact structure $\cald$ has an extremal representative, we mean so up to isotopy. Deforming the K\"ahler form within its K\"ahler class corresponds to deforming the contact structure within its isotopy class. We refer the reader to Lemma 8.1 of \cite{BoTo11} and the discussion immediate above and below this lemma for a more thorough treatment of this issue.

Notice also that under a transverse homothety extremal Sasakian structures stay extremal, and a transverse homothety of a CSC Sasakian structure is also a CSC Sasakian structure. This is because under the transverse homothety $\cals\mapsto \cals_a$ the scalar curvature of the metric $g_a$  is given by (cf. \cite{BG05}, page 228) 
\begin{equation}\label{scalarhomothety}
s_{g_a}=a^{-1}(s_g+2n)-2n. 
\end{equation}
In the Riemannian and K\"ahlerian categories, scaling the metrics scales the scalar curvature $s_g$. So there are essentially 3 types, positive, negative and zero. However, in the Sasakian category, the relevant scaling is the transverse homothety. So the 3 types of scalar curvature in the Sasakian category are $>-2n,<-2n$, and $=-2n$. Moreover, if one has a Sasaki metric of scalar curvature $s_g=-2n$, the entire ray of Sasaki metrics have scalar curvature $-2n$. This is the Sasaki analog of scalar flat K\"ahler metrics.

Following \cite{BGS06} we define the {\it Sasaki extremal set} $\ge(\cald,J)$ to be the subset of $\grk(\cald,J)$ that can be represented by extremal Sasaki metrics.

Before discussing our admissible construction using Hamiltonian 2-forms, we present the constant scalar curvature (CSC) Sasaki metrics that arise from representations of the fundamental group. These involve only cases (1) and (2) of Section \ref{complexsubsec}. Applying the Boothby-Wang construction to the  local product K\"ahler structures in Section \ref{exKmet} and applying Lemma \ref{sasconeprop} gives

\begin{prop}
Let $M^5_{g,k}$ be a contact structure on either $\grS_g\times S^3$ or $\grS_g\tilde{\times} S^3$. Then
\begin{itemize}
\item For the transverse complex structures $J$ belonging to case (1)(a), parameterized by the smooth locus $\calr(\grS_g)^{irr}$ of the $6g-6$ dimensional character variety $\calr(\grS_g)$, {\em every }member of the one dimensional Sasaki cone $\grk(\cald_{k},J)$ admits CSC Sasaki metrics.  

\item For the transverse complex structure $J$ belonging to case (1)(b) {\em no} member of the one dimensional Sasaki cone $\grk(\cald_{k},J)$ admits extremal Sasaki metrics.

\item For the transverse complex structures $J$ belonging to case (2), parametrized by the $2g$ dimensional singular part of the character variety $\calr(\grS_g)$, the regular ray of the two dimensional Sasaki cone $\grk(\cald_{k,(1,1)},J)$ admits CSC Sasaki metrics.
\end{itemize}
\end{prop}

Note that in the last case, case (2) only occurs on the trivial bundle $\grS_g\times S^3$.

\section{Admissible Constructions}\label{admis}

Let $(M,J)$ be a ruled surfaces of type (3).
It follows from Corollary 1 in \cite{ACGT08} that any smooth extremal K\"ahler metric 
on $(M,J)$ must admit a hamiltonian $2$-form of order $1$, that is, be so-called {\it admissible}. In short, this means that - up to scale and biholomorphism - it must arise from the construction described in section \ref{smoothconstruction} . This is a consequence of the uniqueness theorem for extremal K\"ahler metrics in \cite{ChTi05}. However, this uniqueness is still lacking in the orbifold case, so in that case uniqueness holds only within the Hamiltonian 2-form formalism.

Note that, being (local) product metrics, the CSC K\"ahler metrics on ruled surfaces of type (1)(a) and (2) also admit hamiltonian $2$-forms, but now with order $0$. Occasionally we shall refer to these metrics as admissible as well. Thus we may say that any smooth extremal K\"ahler metric on a (geometrically) ruled surface is admissible.

\subsection{The Smooth Case/Regular Ray}\label{smoothconstruction} 

Assume $(M,J)$ equals the total space of ${\mathbb P}({\mathcal O} \oplus {L_{n}}) \rightarrow \Sigma_g$, where $L_n \rightarrow \Sigma_g$ is a holomorphic line bundle of degree $n >0$.

Let us consider a ruled manifold of the form
$S_{n} = {\mathbb P}({\mathcal O} \oplus {L_{n}}) \rightarrow \Sigma_g$,
where $L_{n}$ is a  
holomorphic line bundle of degree $n$, where $n\in {\mathbb Z}^+$ on $\Sigma$, and 
${\mathcal O}$ is the trivial holomorphic line bundle.
Let $g_{\Sigma_g}$ be the
K\"ahler metric
on $\Sigma_g$ of constant scalar curvature $2s_{\Sigma_g}$, with K\"ahler form
$\omega_{\Sigma_g}$, such that
$c_{1}(L_{n}) = [\frac{\omega_{\Sigma_g}}{2 \pi}]$.
That is, $S_{n}$ is the $\bbc\bbp^1$ bundle over $\Sigma_g$ associated to a principle $S^1$ bundle over $\Sigma_g$ with curvature $\omega_{g}$.
Let ${\mathcal K}_{\Sigma_g}$ denote the canonical bundle of $\Sigma_g$. Since $c_1({\mathcal K}_{\Sigma_g}^{-1}) = [\rho_{\Sigma_g}/2\pi]$, where $\rho_{\Sigma_g}$ denotes the Ricci form, we have the relation $s_{\Sigma_g}= 2(1-g)/n$. For each smooth function $\Theta({\gz}), \gz \in [-1,1]$ satisfying
\begin{align}
\label{positivity}
(i)&~ \Theta({\gz}) > 0, \quad -1 < \gz <1,\quad
(ii)\ \Theta(\pm 1) = 0,\quad \\ \notag
(iii)& ~\Theta'(-1) =  2, \quad \Theta'(1) = -2
\end{align}
and each $0<r<1$
we obtain admissible 
K\"ahler metrics;
\begin{equation}\label{metric}
g  =  \frac{1+r {\gz}}{r} g_{\Sigma_g}
+\frac {d{\gz}^2}
{\Theta ({\gz})}+\Theta ({\gz})\theta^2,
\end{equation}
with K\"ahler form 
\begin{equation}
\omega =  \frac{1+r{\gz}}{r}\omega_{\Sigma_g}
+d{\gz}\wedge \theta\,. \label{kf}
\end{equation}
as in \cite{ACGT08} 

Notice that $\gz: S_{n} \rightarrow [-1,1]$ is a moment map of $\omega$ and  the circle action $\tilde{\cala}_{n}(\grl)$ (generated by the vector field $K_{n}$). Futher $\theta$ is a $1$-form such that $\theta(K_{n}) =1$ and $d\theta = \pi^*\omega_{\Sigma_g}$.

As usual $E_n = {\mathbb P}({\mathcal O} \oplus 0)$ denotes the zero section and $F$ denotes the fiber of the bundle $S_n \rightarrow \Sigma_g$. Then $E_n^2 = n$. Note that in the admissible set-up $E_n = \gz^{-1}(1)$.

The K\"ahler class of this metric satisfies
\[
{\rm PD}([\omega]) = 4\pi E_{n}+ \frac{2\pi(1-r)n}{r} F.
\]
Writing $F(\gz) = \Theta(\gz) (1+ r \gz)$, we see from Proposition 1 in \cite{ACGT08} that the corresponding metric is extremal exactly when
$F(\gz)$ is a polynomial of degree at most 4 and $F''(-1/r) = 2rs_{\Sigma_g}$. 
This, as well as the endpoint conditions of \eqref{positivity}, is satisfied precisely when $F(\gz)$ is given by
\begin{equation}\label{extremalpolynomial}
F(\gz) = \frac{(1-\gz^2)h(\gz)}{4 (3 - r^2)},
\end{equation}
where 
\[
\begin{array}{ccl}
h(\gz) & = & (12  - 8 r^2 + 2r^3 s_{\Sigma_g}) \\
\\
 &+& 4r (3 - r^2) \gz \\
 \\
 &+ &  2r^2(2-r s_{\Sigma_g})\gz^2,
 \end{array}
    \]
and $-1<\gz<1$. 
CSC solutions would correspond to $h(\gz)$ being a affine linear function and this is clearly never possible.
When $s_{\Sigma_g} \geq 0$, i.e., $g \leq 1$, we can then check that $h(\gz) >0$ for $-1<\gz <1$ hence $\Theta(\gz)$, as defined via $F(\gz)$ above, satisfies all the conditions of \eqref{positivity}. Thus in this case, for all $r \in (0,1)$ we have an extremal 
K\"ahler metric. However, for $g \geq 2$, i.e., $s_{\Sigma_g} < 0$ the positivity of $h(\gz)$, hence $\Theta(\gz)$ for $-1<\gz<1$ holds only when $0<r<1$ is sufficiently small. 

Assume first that $g \geq 2$, $k \neq 1 \in \bbz^+$, $n=2m$, $m = 1,...,k-1$, and 
$r= \frac{m}{k}$.
Since $ E_{2m}= E_0+mF$  we then have that $[\omega]$ is $4\pi$ times $h + k[\omega_g]$ as in \eqref{symclass}. 
With $r=m/k$ and $s_{\Sigma_g} = (1-g)/m$ we have
\[
\begin{array}{ccl}
h(\gz) & = & 2 (6 k^3 + m^2 - g m^2 - 4 k m^2)/k^3 \\
\\
 &+& (4 m (3 k^2 - m^2)/k^3) \gz \\
 \\
 &+ &  (2 ( g + 2 k-1) m^2/k^3)\gz^2.
 \end{array}
    \]
Let us explore the positivity of $h(\gz)$ for this case.

Clearly $h(\gz)$ is  a concave up parabola in $\gz$ and the absolute minimum of this parabola
is easily calculated to be $2 M/(k^3 (g+2 k-1))$, where
\[
M = 3 k^3 ( 2 g + k-2) +( 2 g - g^2 + 6 k - 6 g k - 2 k^2-1)m^2  -m^4
\]
For a fixed $g\geq 2$ and $k\geq2$, this is a decreasing function of $m \in [1,k-1]$, so for
$m=1,...,k-1$, $M$ would be minimized when $m=k-1$. Now
\[
M|_{m=k-1} =  12 k - 21 k^2 + 8 k^3-2  + 2 (1 - 5 k + 7 k^2) g -( k-1)^2g^2
\]
and for a fixed $k\geq 2$, this is a concave down parabola in $g$ and it is easy to verify that  it is positive for any $k \geq2$ at $g=2$ and $g=19$ and hence at any $2 \leq g \leq 19$.
Thus for any $k\geq 2$, $m = 1,...,k-1$ and $g = 2,...,19$, the minimum of $h(\gz)$ is positive and thus the positivity of $h(\gz)$ is satisfied. Finally, for any fixed $g \geq 2$, we see that
for sufficiently large values of $k$, $M|_{m=k-1}  >0$. Thus for any given $g \geq 2$, there exist some $k_g \geq 2$, s.t. $\forall k \geq k_g$ and $m=1,...,k-1$ the minimum of $h(\gz)$ is positive and thus the positivity of $h(\gz)$ is satisfied.

On the other hand for e.g. $k=6$, and $m=k-1=5$ we have that
$M = 1042 + 446 g - 25 g^2 < 0$ for all $g \geq 20$, so the minimum of $h(\gz)$ is negative for
$g \geq 20$ and moreover the location of the minimum is $z=-83/(5 (11 + g)) \in (-1,1)$. Therefore
the positivity condition for $h(\gz)$ is NOT satisfied for these data.

The case $g \geq 2$, $k \in \bbz^+$, $n=2m+1$, $m = 0,...,k-1$, and 
$r= \frac{2m+1}{2k+1}$ is completely similar. More specifically the details of the argument are as follows: 
Since $E_{2m+1} = E_1 + m F$  we still have that $[\omega]$ is $4\pi$ times $h + k[\omega_g]$ as in \eqref{symclass}. 
With $r=(2m+1)/(2k+1)$ and $s_{\Sigma_g} = 2(1-g)/(2m+1)$ we have
\[
\begin{array}{ccl}
h(\gz) & = & 4 (2 - g + 14 k + 36 k^2 +24 k^3 - 4 m(1+m)( (1+g) +4k) )/(1 + 2 k)^3 \\
\\
 &+& 8 (1 + 2 m) (1 + 6 k(1+k)  - 2 m(1+m))/(1 + 2 k)^3 \gz \\
 \\
 &+ &  4 (g + 2 k) (1 + 2 m)^2/(1 + 2 k)^3\gz^2.
 \end{array}
    \]
Again $h(\gz)$ is  a concave up parabola in $\gz$ and the absolute minimum of this parabola
is $4 M/((1 + 2 k)^3 (g + 2 k))$, where
\[
\begin{array}{ccl}
M & = & -(1- g)^2 - 4 k(2+5k+3k^3) + 12 g k(1  + 3 k + 2  k^2)  \\
\\
& + & 4(1 - g(1+g +6k)  - 2 k(k-2)) m -  4(g(1 + g)  + 2 k( 3g + k-2) m^2 \\
    \\
    & - & 8 m^3 - 4 m^4
    \end{array}
\]
For a fixed $g\geq 2$ and $k\geq2$, this is a decreasing function of $m \in [1,k-1]$, so for
$m=0,...,k-1$, $M$ would be minimized when $m=k-1$. Now
\[
M|_{m=k-1} =  -1 - 12 k - 36 k^2 + 32 k^3  + 
 (2 + 16 k + 56 k^2)g - ( 4 k^2- 4 k+1 ) g^2
\]
and for a fixed $k\geq 2$, this is a concave down parabola in $g$ which is positive at $g=2$ and $g=19$ and hence at any $2 \leq g \leq 19$.
Thus for any $k\geq 2$, $m = 1,...,k-1$ and $g = 2,...,19$, the minimum of $h(\gz)$ is positive and thus the positivity of $h(\gz)$ is satisfied. Finally, for any fixed $g \geq 2$, we see that
for sufficiently large values of $k$, $M|_{m=k-1}  >0$. Thus for any given $g \geq 2$, there exist some $k_g \geq 2$, s.t. $\forall k \geq k_g$ and $m=0,...,k-1$ the minimum of $h(\gz)$ is positive and thus the positivity of $h(\gz)$ is satisfied.

Again for e.g. $k=6$, and $m=k-1=5$ we have that
$M = 5543 + 2114 g - 121 g^2 < 0$ for all $g \geq 20$, so the minimum of $h(\gz)$ is negative for
$g \geq 20$ and moreover the location of the minimum is $z=-(193/(11 (12 + g))) \in (-1,1)$. Therefore
the positivity condition for $h(\gz)$ is NOT satisfied for these data.

Putting this together with results from Section \ref{regsas} we have

\begin{theorem}\label{deg>0ext}
For any choice of genus $g=2,...,19$ the regular ray in the Sasaki cone $\grk(M^5_{g,l,\bfw},J_\bfw)$  admits an extremal representative with non-constant scalar curvature.

For any choice of genus $g=20,21,...$ there exists a $K_g \in \bbz^+$ such that
if $l |\bfw | \geq K_g$,
then the regular ray in the Sasaki cone $\grk(M^5_{g,l,\bfw},J_\bfw)$ admits an extremal representative with non-constant scalar curvature.

For any choice of genus $g=20,21,...$ there exist at least one choice of $(l,w_1,w_2)$ such that the regular ray in the Sasaki cone $\grk(M^5_{g,l,\bfw},J_\bfw)$ admits no extremal representative, despite the fact that
the quasi-regular Sasaki structure $\cals_{l,\bfw}=(\xi_\bfw,\eta_{l,\bfw},\Phi_\bfw,g_\bfw)$ is extremal.
\end{theorem}

This proves Theorem \ref{deg>0intro} of the Introduction.

\begin{remark}\label{non-existence-example}
Notice that by Lemma \ref{evenHir} the choice of $k=6$ and $m=5$ alluded to above is manifested in either the choice of
$(l,w_1,w_2) = (1,11,1)$, where the Sasaki manifold is $\Sigma_g \times S^3$, or in the choice of $(l,w_1,w_2) = (1,12,1)$, where the Sasaki manifold is $\Sigma_g \tilde{\times} S^3$.
\end{remark}

\subsection{Orbifold Singularities:}\label{orbisetup}

We will now extend the constructions in Section \ref{smoothconstruction} to a certain orbifold case Let $p,q \in {\bbz}^+$ be relative prime and consider the
(simply connected) weighted projective space $\bbc\bbp_{p,q}$. Generalizing the conditions \eqref{positivity} above to
 \begin{align}
\label{conesingpositivity}
(i)&~ \Theta({\gz}) > 0, \quad -1 < \gz <1,\quad
(ii)\ \Theta(\pm 1) = 0,\quad \\ \notag
(iii)& ~\Theta'(-1) =  2/p, \quad \Theta'(1) = -2/q
\end{align}
- and keeping everything else the same as above -
we produce a K\"ahler orbifold metric on the $\bbc\bbp_{p,q}$ bundle over $\Sigma_g$ associated to a principle $S^1$ bundle over $\Sigma_g$ with curvature $\omega_{\Sigma_g}$.
This is of course an orbifold pseudo-Hirzebruch surface $(S_n, \Delta_{(q,p)})$.
We will also allow the possibility that $n <0$, which means that $-1<r<0$, and $-\omega_{\Sigma_g}$ is a K\"ahler form with scalar curvature $-2 s_\Sigma = -2(1-g)/n$.

\begin{remark}\label{scalerem}
Via the orbifold Boothby-Wang construction, for appropriate choices of $r$, the metrics lifts to Sasaki metrics on $S^1$ orbi-bundles over $M$. Making a transverse homothety transformation on the Sasakian level as in Equation 7.3.10 of \cite{BG05} rescales the Reeb vector field by a factor $a^{-1}$ and the transverse K\"ahler metric by a factor $a$. Keeping the vector field $K_{n}$ (and the corresponding $\theta$) fixed  this necessitates a rescaling of the endpoints of the interval of the moment map. So with the range of $z$ fixed to be $[-1,1]$ 
transversal homotheties in the Sasaki cone are not visible in the above construction.

On the other hand, if for some $\kappa\in \bbz^+$ we formally replace $(p,q)$ by $(\kappa p,\kappa q)$ and $s_\Sigma$ by 
$\frac{s_\Sigma}{\kappa}$ then $\Theta(\gz)$ simply rescales by $1/\kappa$. To interpret this correctly, we replace $K_{n}$ by $\frac{1}{\kappa}K_{n}$, $\theta$ by $\kappa\theta$, $\omega_{\Sigma_g}$ by $\kappa\omega_{\Sigma_g}$, and leave $r$ and $z$ unchanged. Then the new metric (which locally is just $\kappa$ times the old) is a K\"ahler metric on the total space of the $\bbc\bbp_{p,q}/\bbz_{\kappa}$ bundle over $\Sigma_g$ associated to a principle $S^1$ bundle over $\Sigma_g$ with curvature $\kappa\omega_{\Sigma_g}$. Thus the two 
K\"ahler orbifolds are related up to a $\kappa$-fold covering and a homothety.
 \end{remark}
 
 \bigskip

The conditions for extremality remains the same and
together with the endpoint conditions of \eqref{conesingpositivity} this implies that $F(\gz) = \Theta(\gz)(1+r\gz)$ must be given by
\begin{equation}\label{orbiextremalpolynomial}
F(\gz) = \frac{(1-\gz^2)h(\gz)}{4 p q (3 - r^2)},
\end{equation}
where 
\[
\begin{array}{ccl}
h(\gz) & = & q (6 - 3 r - 4 r^2 + r^3) + p (6 +  3 r - 4 r^2 - r^3) \\
\\
 &+& 2 (3 - r^2) (q (r-1) + p (1 + r))\gz \\
 \\
 &+ &  r(p (3 + 2 r - r^2) - q (3 - 2 r - r^2))\gz^2\\
 \\
 &+ &2pqr^3 s_{\Sigma_g}(1-\gz^2),
 \end{array}
    \]
and $-1<\gz<1$. When $rs_{\Sigma_g} \geq 0$, i.e., $g \leq 1$, we can then check that $\Theta(\gz)$ as defined via $F(\gz)$ above satisfies all the conditions of \eqref{conesingpositivity}. Again, for $g \geq 2$, i.e., $rs_{\Sigma_g} < 0$ the positivity of $h(\gz)$, hence $\Theta(\gz)$ for $-1<\gz<1$ does not necessarily hold. On the other hand, by design $h(\gz)$ is such that when $F(\gz)$ is given by
\eqref{orbiextremalpolynomial} and $\Theta(\gz) = F(\gz)/(1+r \gz)$ then the endpoint conditions of \eqref{conesingpositivity} are satisfied. Therefore we always have that $h(\pm 1) >0$. In particular,
if either the graph of $h(\gz)$ is  concave down or $h(\gz)$ is an affine linear function then the positivity is certainly satisfied. The latter scenario is - in contrast with the smooth case - now quite possible (see Section \ref{examples} below) and provides examples of CSC K\"ahler orbifold metrics.

\subsection{CSC Examples}\label{examples}
From the above section it is clear that CSC examples arise when we can solve
\begin{equation}\label{CSCequation}
p (3 + 2 r - r^2) - q (3 - 2 r - r^2)-2pqr^2 s_{\Sigma_g}=0
\end{equation}

We now need to make the connection with the quasi-regular Sasaki structures obtained in section 
\ref{quasi-reg-ray}. Consider the Sasaki cone $\grk(M^5_{g,l,\bfw},J_\bfw)$  for some
$l \in  {\mathbb Z}^+$ and co-prime $w_1 \geq w_2 >0 \in {\mathbb Z}$ and consider the ray determined by $\bfv=(v_1,v_2)$ where $v_1,v_2\in\bbz^+$ and  $\gcd(v_1,v_2)=1$.  From Lemma
\ref{quasireglem} we know that the base space $B_{\bfv,\bfw}$ is an orbifold pseudo-Hirzebruch surface $(S_n, \Delta_{\bf v})$ with $n=l(w_1v_2 - w_2 v_1)$ and the
 K\"ahler class is $[\omega] = h + k[\omega_g]$ with $k$ given by Lemma \ref{qregk}.
 We will assume that $(v_1,v_2) \neq (w_1,w_2)$, so $n\neq 0$, noting that the ray determined by $(w_1,w_2)$ is certainly extremal (and CSC if and only if $(w_1,w_2)=(1,1)$).
 As in section \ref{smoothconstruction}, if $n=2m$ is even, we need $r=m/k$ and if $n=2m+1$ is odd we need $r=(2m+1)/(2k+1)$. Either way this implies that 
 $$ r= \frac{w_1v_2 - w_2 v_1}{w_1v_2 + w_2 v_1}.$$ Further, it is clear that $p=v_2$ and $q=v_1$ (we remind the reader that $E_n= \gz^{-1}(1)$). Finally, $s_{\Sigma_g} = 2(1-g)/(l(w_1v_2 - w_2 v_1))$.
 Using this, equation \eqref{CSCequation} simplifies to
$$ lw_2^2 v_1^3+(g-1 +2lw_1)w_2v_1^2v_2 + (1-g-2lw_2) w_1 v_1 v_2^2-lw_1^2v_2^3 =0$$
If we set $v_2=c v_1$ (so $c \neq 1$ is assumed), then this in turn simplifies to
 \begin{equation}\label{cscallgenera}
  lw_1^2 c^3 - (1-g-2lw_2) w_1c^2-(g-1 +2lw_1)w_2c- lw_2^2=0.
  \end{equation} 
It is a simple calculus exercise to show that for any pair of positive integers $w_1>w_2$, equation 
\eqref{cscallgenera}
has precisely one positive solution $c=c_{\bf w}\neq 1$. Thus in this case the Sasaki cone $\grk(M^5_{g,l,\bfw},J_\bfw)$ has a unique ray, $v_2=c_{\bf w}v_1$, of CSC Sasaki metrics. In fact, one can also easily verify that since $w_1 > w_2$, then $c_{\bf w} \in (w_2/w_1,1)$ and hence for the CSC ray $(\frac{w_2}{w_1}) v_1 < v_2 < v_1$. Whenever the co-prime pair $(w_1,w_2)$ is such that $c_{\bf w}$  is rational, then it is possible to chose a corresponding pair
$(v_1,v_2)$ of co-prime positive integers and the CSC Sasaki metrics are quasi-regular. If $c_{\bf w}$ is non-rational then we have a ray of irregular CSC Sasaki metrics. Also note that if we let $w_1=w_2=1$, then of course we have that the ray $v_2=v_1$ is a regular CSC ray. In this case  equation \eqref{CSCequation} has no positive solution $c_{\bf w}$ other than $c_{\bf w} =1$ manifesting the uniqueness of the CSC ray also in this case.

\begin{proposition}
For genus $g \geq 1$, the Sasaki cone $\grk(M^5_{g,l,\bfw},J_\bfw)$ contains a unique ray of CSC Sasaki metrics such that the transverse K\"ahler metric admits a hamiltonian $2$-form. 
\end{proposition}

This proposition proves Theorem \ref{thm1} of the Introduction.

\begin{example}\label{genus1csc}
Assume $g=1$. In this case \eqref{cscallgenera} simplifies to
\begin{equation}\label{cscgenus1}
w_1^2 c^3 + 2 w_1 w_2 c^2 -2 w_1 w_2 c - w_2^2 = 0
\end{equation}
For instance for any integer $ t>1$, choosing $(w_1,w_2)$ such that 
$w_2/w_1 = (1+2t)/(t^2(2+t))$ yields $c_{\bf w} = (1+2t)/(t(2+t))$. On the other hand for e.g 
$(w_1,w_2) = (2,1)$, equation \eqref{cscgenus1} is equivalent to
$$4c^3+4c^2-4c-1,$$
which has no rational roots.
\end{example}

\begin{example}\label{genushighercscfamily}
Assume $g \geq 2$. Then it is easy to see that for $l=g-1$, $w_1=12$ and $w_2=1$, equation
\eqref{cscallgenera}  is solved for $c_{\bf w} = 1/3$.
\end{example}

Using Examples \ref{genus1csc} and \ref{genushighercscfamily} we arrive at

\begin{proposition}
For any genus $g \geq 1$, there exist values of $l$ and ${\bf w} = (w_1,w_2)\neq (1,1)$ such that 
the Sasaki cone $\grk(M^5_{g,l,\bfw},J_\bfw)$ admits a quasi-regular CSC Sasaki metric.
\end{proposition}

\begin{example}\label{highergenuscsc}
Assume $w_1=12$ while $w_2=1$ and $l=1$. 
Then \eqref{cscallgenera} is
\begin{equation}\label{w1=12w2=1}
144 c^3 + 12(g+1) c^2 - (23+g) c -1 = 0,
\end{equation}
and this has a rational solution $c_{\bf w}=1/3$ when $g=2$ (special case of Example \ref{genushighercscfamily})
and a rational solution  $c_{\bf w}=1/6$ when $g=23$. The rational root test reveals that for all other values of $g$ the solution $c_{\bf w}$ is irrational.
It is interesting to observe that 
$\lim_{g\rightarrow +\infty} c_{\bf w} = 1/12$ corresponding to $(v_1,v_2)= (12,1)=(w_1,w_2)$.
This is not surprising for a couple of reasons. Firstly, from Remark \ref{non-existence-example} we know that for $g \geq 20$, the regular ray in the Sasaki cone $\grk(M^5_{g,l,\bfw},J_\bfw)$ is not extremal, while certainly the ray $(v_1,v_2)= (12,1)=(w_1,w_2)$ is extremal. By the Openess Theorem \cite{BGS06} any ray {\em sufficiently close} to the ray  $(v_1,v_2)= (12,1)=(w_1,w_2)$ is also extremal. Intuitively it makes sense that the larger the $g$, the closer the rays need to be to $(v_1,v_2)= (12,1)=(w_1,w_2)$ in order to be extremal and so the CSC ray must be {\em pushed towards} 
$(v_1,v_2)= (12,1)=(w_1,w_2)$ as $g$ increases. Secondly, while the product metric 
$\omega_g + \omega_{\bf w}$ on  
$\Sigma_g \times \bbc\bbp_{w_1,w_2}$ is extremal but not CSC, the larger the value of $g$ is, the more that the scalar curvature of $\omega_g$, which is constant, dominates the overall scalar curvature
of the product which then in turn becomes ``approximately constant''. 
\end{example}

\subsection{Extremal Metrics in the Sasaki cone}\label{extremalincone}
From the observations in section \ref{orbisetup} and by the Openess Theorem \cite{BGS06} we have that for genus $g=1$, the Sasaki cone $\grk(M^5_{g,l,\bfw},J_\bfw)$  is exhausted by extremal Sasaki metrics. On the other hand we know from Theorem \ref{deg>0ext} that for $g \geq 20$ there exists at least one choice of $(l,w_1,w_2)$ for which the regular ray in  the Sasaki cone $\grk(M^5_{g,l,\bfw},J_\bfw)$ is not extremal. Now $h(\gz)$ from section \ref{orbisetup}
with  $ r= \frac{w_1v_2 - w_2 v_1}{w_1v_2 + w_2 v_1}$, $p=v_2$,  $q=v_1$, \newline
$s_{\Sigma_g} = 2(1-g)/(l(w_1v_2 - w_2 v_1))$, and finally $v_2=c v_1$ can be written as

$$
\begin{array}{cl}
&\frac{l(cw_1 + w_2)^3}{4 v_1}h(\gz) \\
\\
 = & l w_2^3 (1 - \gz)^2\\
\\
 +& w_2^2 (1-\gz)(7lw_1+1-g +(1-g-lw_1)\gz) c \\
\\
+ & 2w_1w_2(g-1+2l(w_1+w_2) + l(w_2-w_1)\gz +(1-g-l(w_1+w_2)) \gz^2) c^2\\
\\
 +& w_1^2 (1 + \gz) (1 - g + 7 l w_2  +(g-1 + l w_2) \gz)c^3 \\
\\
+& lw_1^3(1+\gz)^2 c^4.
\end{array}
$$
Viewing $\frac{l(cw_1 + w_2)^3}{4 v_1}h(\gz)$ as a polynomial of $c$ for a moment we observe that the constant coefficient and the coefficient of $c^4$ are both positive for $\gz \in (-1,1)$.
Further, for any value of genus $g$, the coefficient of $c^2$ is a concave down polynomial in $\gz$ which is positive at $\gz=\pm1$ and hence positive for all $\gz\in[-1,1]$. Finally the coefficients of $c$ and $c^3$ are both positive for $\gz \in (-1,1)$ as long as $g \leq Min(1+3lw_1,1+3lw_2)$.
Thus, the latter condition guarantees that $h(\gz) >0$ for $z \in (-1,1)$.

Therefore, if this condition is satisfied, then any rational $c \in (0,+\infty)$ corresponds to a pair of co-prime positive integers $(v_1,v_2)$ which in turn yields a quasi-regular extremal Sasaki metric and by the Openess Theorem \cite{BGS06} and denseness of quasi-regular rays in $\grk(M^5_{g,l,\bfw},J_\bfw)$ we may conclude:
\begin{proposition}\label{exhprop}
Consider the Sasaki cone $\grk(M^5_{g,l,\bfw},J_\bfw)$  for some
$l \in  {\mathbb Z}^+$ and co-prime $w_1 \geq w_2 >0 \in {\mathbb Z}$. 
If $g \leq 1+3lw_2$, then $\ge(M^5_{g,l,\bfw},J_\bfw)= \grk(M^5_{g,l,\bfw},J_\bfw)$, that is the Sasaki cone is exhausted by extremal Sasaki metrics. In particular for $g \leq 4$ and for any $l,{\bf w}$ the Sasaki cone $\grk(M^5_{g,l,\bfw},J_\bfw)$ is exhausted by extremal Sasaki metrics
whose transverse K\"ahler metric admits hamiltonian $2$-forms.
\end{proposition}

\begin{example} 
Let us consider the simplest of all cases, namely $(w_1,w_2)=(1,1)$. That is the regular ray is CSC. From Proposition \ref{exhprop} it follows that for $l>1$, the Sasaki cone is exhausted by (admissible) extremal Sasaki metrics for
$g \leq 7$. One may check that for $l=1$ and $g$ equal to $5$ or $6$, the 2nd order polynomial 
$h(\gz)$ is positive for all choices  of $(v_1,v_2)$, so together with Proposition \ref{exhprop} again, this tells us that for $(w_1,w_2)=(1,1)$ and any genus $0<g\leq 6$ the Sasaki cone is exhausted by (admissible) extremal Sasaki metrics. On the other hand for
$(w_1,w_2)=(1,1)$, $l=1$, and $g=7$, we can check that positivity of $h(\gz)$ fails for 
$c=v_2/v_1$ exactly when  $c \in (0,13-2\sqrt{42}) \cup (13+2\sqrt{42},+\infty)$, so the Sasaki cone is not exhausted by admissible extremal metrics in this case.
\end{example}

\begin{proof}[Proof of Theorem \ref{sasexh}]
Using Lemma \ref{evenHir} the estimate in Proposition \ref{exhprop} can be rewritten as 
$$k\geq m+\frac{g-1}{3}.$$
The theorem now follows from Proposition \ref{exhprop}.
\end{proof}

We know from Theorem \ref{deg>0ext} that in general the Sasaki cone is not exhausted by extremal Sasaki metrics.
It is also easy to see that for fixed $(l,w_1, w_2, c \neq w_2/w_1)$, 
$$\lim_{g \rightarrow + \infty} h(0) = - \infty$$
and hence positivity of $h(\gz)$ for $z \in (-1,1)$ fails for large values of $g$. Thus
the part of the cone which have extremal Sasaki metrics
whose transverse K\"ahler metric admits hamiltonian $2$-forms seems to ``shrink''  as $g \rightarrow +\infty$.

Note that if we view $\frac{l(cw_1 + w_2)^3}{4 v_1}h(\gz)$ as a second order polynomial in $\gz$, then the coefficient of $\gz^2$ is $(cw_1-w_2)$ times the left hand side of \eqref{cscallgenera}. Thus for $w_2/w_1 < c < c_{\bf w}$ we have that the graph of $h(\gz)$ is a concave down parabola and since $h(\pm 1) >0$, the positivity of $h(\gz)$ for $\gz \in (-1,1)$ follows.

\begin{proposition}
Consider the Sasaki cone $\grk(M^5_{g,l,\bfw},J_\bfw)$  for some
$l \in  {\mathbb Z}^+$ and co-prime $w_1 \geq w_2 >0 \in {\mathbb Z}$. 
The rays in $\grk(M^5_{g,l,\bfw},J_\bfw)$ which are between the ray
$v_2 = (w_2/w_1)v_1$ and the CSC ray $v_2=c_{\bf w} v_1$ are all extremal.
Thus the two rays belong to the same connected component of the  {\it Sasaki extremal set} $\ge(\cald,J_\bfw)$ in $\grk(M^5_{g,l,\bfw},J_\bfw)$.
\end{proposition}
This proves Theorem \ref{connthm}. By the Openess Theorem \cite{BGS06}, the rays sufficiently close to this ``wedge'' are also extremal.

\begin{example}
Let us return to example \ref{highergenuscsc}, that is, let $l=1$, $w_1=12$ and $w_2=1$. Further let $g=23$ so that we know we have a CSC ray for $c=1/6$.
In this case we have
$$
\begin{array}{ccl}
\frac{l(1+12 c)^3}{4 v_1}h(\gz) & = & (1 - \gz)^2\\
\\
& +&  2(1-\gz)(31 -17\gz) c \\
\\
&+ & 24(48-11\gz -35 \gz^2) c^2\\
\\
& +& 144 (1 + \gz) (-15+23 \gz)c^3 \\
\\
&+&1728(1+\gz)^2 c^4\\
\\
&=& 1 + 62 c + 1152 c^2 - 2160 c^3 + 1728 c^4\\
\\
&+& 2 (12 c^2-1) (1 + 48 c + 144 c^2)\gz \\
 \\
& +&  ( 6 c-1) (12 c-1) (1 + 
    52 c + 24 c^2)\gz^2.
\end{array}
$$
Note that $c=1/12$ corresponds to the case when $(v_1,v_2)=(w_1,w_2)$ where the admissible set-up and hence $h(\gz)$ is not defined. The ray $(v_1,v_2)=(w_1,w_2)$ is extremal since the product metric 
$\omega_g + \omega_{\bf w}$ on  
$\Sigma_g \times \bbc\bbp_{w_1,w_2}$ is extremal. When $c=1$ we know, since $g \geq 20$, that the (regular) ray is not extremal. To determine which rays in the rest of the Sasaki cone are extremal with extremal metrics arising from the admissible construction, we will now investigate for which
$c\in (0,+\infty)$ the second order polynomial
$$
\begin{array}{ccl}
p(\gz)& = & 1 + 62 c + 1152 c^2 - 2160 c^3 + 1728 c^4\\
\\
&+ &2 (12 c^2-1) (1 + 48 c + 144 c^2)\gz  +
( 6 c-1) (12 c-1) (1 +  52 c + 24 c^2)\gz^2,
\end{array}$$
hence $h(\gz)$ is positive on the interval $(-1,1)$.

First, notice that $p(\pm 1) >0$ for all $c>0$ and 
$p'(1) = 4 c (-7 - 486 c + 1944 c^2 + 1728 c^3) \leq 0$ for $c \in (0,\tilde{c}]$, where $\tilde{c} \approx 0.22$. 
Therefore, for $c \leq \tilde{c}$,
$p(\gz) >0$ for $-1 < \gz < 1$.
Since $\tilde{c} > 1/6$, we know that for $c>\tilde{c}$, $p(\gz)$ is concave up. 
Moreover, since 
$p'(-1) = -4-4 c (41 - 354 c + 1368 c^2) <0$ for all $c>0$ and $p'(1) > 0$ for all $c > \tilde{c}$, we know that, for $c > \tilde{c}$, if $p(\gz)$ has any root for, then this root is inside the interval $(-1,1)$ and thus,
for $c >\tilde{c}$,
$p(\gz) >0$ for $-1 < \gz < 1$ if and only if the discriminant
$$
D(c) = 16 c^2 (37 + 5820 c + 197748 c^2 - 1528416 c^3 + 1622592 c^4)
$$
of $p(\gz)$ is negative. Now one can verify numerically that $D(c) < 0$ for $\tilde{c} < c < \hat{c}$ and $D(c) \geq 0$ for $c \geq \hat{c}$, where $\hat{c} \approx 0.78$. Thus $p(\gz) >0$ for $-1 < \gz < 1$ if and only if $0< c <\hat{c}$. Hence, in this case, a ray $v_2=c v_1$ in the Sasaki cone is extremal  with extremal metrics arising from the admissible construction if and only if $0< c <\hat{c}$. 
\end{example}

\subsection{Scalar Curvature}
For a given Sasaki cone $\grk(M^5_{g,l,\bfw},J_\bfw)$ assume we have a choice of co-prime
$(v_1,v_2)$ such that the corresponding ray is extremal and that the transverse K\"ahler metric admits a hamiltonian $2$-form of order $1$. That is, the transverse K\"ahler metric is $(1/{4\pi})$ times an admissible extremal K\"ahler metric \eqref{metric} with 
$s_{\Sigma_g} = 2(1-g)/(l(w_1v_2 - w_2 v_1))$,  
$\Theta(\gz) = \frac{F(\gz)}{(1+r\gz)}= \frac{(1-\gz^2)h(\gz)}{4pq (3 - r^2)(1+r\gz)}$, 
 $ r= \frac{w_1v_2 - w_2 v_1}{w_1v_2 + w_2 v_1}$ and 
$h(\gz)$ given as in the beginning of Section \ref{orbisetup} (where $p=v_2$ and $q=v_1$).
For example from Equation (10) in \cite{ACGT08} it then follows that the scalar curvature $s_g^T$ (which should be a linear function of $\gz$)  of the transverse K\"ahler metric equals
$$ 4\pi\left (\frac{2 r s_{\Sigma_g} -F''(\gz)}{(1+r\gz)} \right ) = A_{l,{\bf w},{\bf v}} + B_{l,{\bf w},{\bf v}}\,  \gz ,$$
where
$$A_{l,{\bf w},{\bf v}} =\frac{24 \pi (v_1 w_2 (1 - g + l w_1) + v_2 w_1 (1 - g + l w_2))}{l (v_2^2 w_1^2 + 4 v_1 v_2 w_1 w_2 + v_1^2 w_2^2)}$$
and 
$$B_{l,{\bf w},{\bf v}}=\frac{24 \pi (l v_2^3 w_1^2 - v_1^2 v_2 w_2(g - 1 + 2 l w_1)  - l v_1^3 w_2^2 + 
  v_1 v_2^2 w_1 ( g -1+ 2 l w_2))}{l v_1 v_2 (v_2^2 w_1^2 + 4 v_1 v_2 w_1 w_2 + v_1^2 w_2^2)}.$$
Then $B_{l,{\bf w},{\bf v}}=0$ corresponds to the scalar curvature being constant.

\begin{example}\label{genus1cscvalue}
For the CSC family in Example \ref{genus1csc} we calculate that the transverse constant scalar curvature 
$s_g^T = \frac{24\pi t}{(1+2t)v_1}$ and thus the Sasakian scalar curvature $s_g =  \frac{24\pi t}{(1+2t)v_1}-4$.
\end{example}

\begin{example}\label{genushighercscfamilyvalue}
For the CSC family in Example \ref{genushighercscfamily}, the tranverse constant scalar curvature does not depend on the genus $g \geq 2$. In fact,  $s_g^T = 8\pi/3$, and so $s_g =  8\pi/3-4$.
\end{example}

\begin{example}\label{highergenuscscvalue}
In Example \ref{highergenuscsc} we saw that for genus $g=23$, $l=1$, $w_1=12$, $w_2=1$, $v_1=6$, and $v_2=1$ we have a CSC metric. Here we calculate that the transverse constant scalar curvature $s_g^T = -16\pi$, and so $s_g =   -16\pi-4$.
\end{example}

\bigskip 

These examples show that both of the ranges for scalar curvature, $s_g>-4$ and $s_g<-4$, as discussed near the end of Section \ref{ess}, are realized. Moreover, the range $s_g>-4$ is realized for all genera.

If a Sasaki metric on $\Sigma_g \times S^3$ and $\Sigma_g \tilde{\times} S^3$ has constant scalar curvature $s_g=-4$, then this scalar curvature is invariant under transverse homothety. That is, the entire ray of Sasaki metrics determined by this Sasaki metric has constant scalar curvature $s_g =-4$. Correspondingly the transverse K\"ahler metrics are scalar flat;  $s_g^T =0$. We now present some examples of such Sasaki metrics.

First the Boothby-Wang construction over $\Sigma_g\times\bbc\bbp^1$ with $g\geq 2$, taking scalar curvatures $-a$ and $a$ on the two factors clearly yields regular Sasaki metrics with $s_g = -4$ on $\Sigma_g\times S^3$ for any $a>0$. These metrics cannot be equivalent since $\Sigma_g\times \bbc\bbp^1$ does not admit a homothety diffeomorphism.

For a general Sasaki cone $\grk(M^5_{g,l,\bfw},J_\bfw)$ with ${\bf w} \neq (1,1)$ and $g \geq 2$, we can investigate if the CSC rays we have discovered might have scalar curvature $-4$, that is, whether the transverse admissible K\"ahler structure has scalar curvature $s_g^T= 0$. Note that this would be equivalent to 
$$A_{l,{\bf w},{\bf v}} =B_{l,{\bf w},{\bf v}} =0.$$
Looking at the formula for $A_{l,{\bf w},{\bf v}}$ we see right away that this cannot vanish for genus $g=2$. Let us hence assume that $g >2$. 
It is easy to see that $A_{l,{\bf w},{\bf v}} =0$ if and only if
$$v_2/v_1= \frac{(lw_1 - (g-1))w_2}{((g-1)-lw_2)w_1},$$
where $lw_2 < g-1$ is a necessary requirement. Then the condition $B_{l,{\bf w},{\bf v}} =0$
(i.e. $v_2/v_1= c_{\bf w}$) becomes
$$\left (l^2 w_1 w_2 - (g-1)^2 \right)\left(l^2 w_1^2 + 2l(g-1-2lw_2) w_1 +l^2w_2 - 2(g-1)( l w_2 + (g-1)\right)=0.$$
By using the quadratic formula one easily sees that
$$l^2 w_1^2 + 2l(g-1-2lw_2) w_1 +l^2w_2 - 2(g-1)( l w_2 + (g-1))=0$$ has no solutions 
for $w_1,w_2,l ,g \in \bbz^+$ with $lw_2 < g-1$.
Thus $s_g^T = 0$ comes down to the equation
$$l^2 w_1 w_2 - (g-1)^2 = 0,$$
which can be solved by $(l,w_1,w_2) = (1, (g-1)^2,1)$. If $(g-1)$ is not a prime number there are other possibilities as well.  

If $g \geq 3$ is odd, the above solution yields a ray of quasi-regular Sasaki metrics on $\Sigma_g \tilde{\times} S^3$ with constant scalar curvature $s_g =-4$. To see this, note that $l(w_1+w_2) = (g-1)^2+1$ is odd and use Theorem \ref{diffeotype2}. On the other hand, if  $g=2k+1 \geq 5$ is odd, then we can choose  $(l,w_1,w_2)=(2,k^2,1)$ as another solution and now $l(w_1+w_2) = 2(k^2+1)$ is even,
yielding quasi-regular Sasaki metrics on $\Sigma_g {\times} S^3$ with constant scalar curvature $s_g =-4$. 

If $g \geq 4$ is even, then $(g-1)$ is odd and so any solution $(l,w_1,w_2)$ of $l^2 w_1 w_2 - (g-1)^2 = 0$ must satisfy that
$l^2 w_1 w_2$ is odd and so in particular $w_1$ and $w_2$ are both odd and hence
$l(w_1+w_2)$ is even. This means that these solutions give quasi-regular Sasaki metrics with constant scalar curvature $s_g =-4$ on $\Sigma_g {\times} S^3$. 
Together with the regular Sasaki metrics with constant scalar curvature $s_g =-4$ coming from the Boothby-Wang construction we then have the following conclusion.
\begin{proposition}
For genus $g \geq 2$ there are regular Sasaki metrics with constant scalar curvature $s_g =-4$ on $\Sigma_g {\times} S^3$.
For genus $g= 3$ there are quasi-regular Sasaki metrics with constant scalar curvature $s_g =-4$
on $\Sigma_g \tilde{\times} S^3$.
For genus $g \geq 5$ and odd there are quasi-regular Sasaki metrics with constant scalar curvature $s_g =-4$
on $\Sigma_g \tilde{\times} S^3$ as well as $\Sigma_g {\times} S^3$. 
For genus $g \geq 4$ and even there are quasi-regular Sasaki metrics with constant scalar curvature $s_g = - 4$ on $\Sigma_g {\times} S^3$.
\end{proposition}
\begin{remark}
It is possible that another construction would yield quasi-regular Sasaki metrics with constant scalar curvature $s_g =-4$ in the cases we seem to be missing above.
\end{remark}

\def\cprime{$'$} \def\cprime{$'$} \def\cprime{$'$} \def\cprime{$'$}
  \def\cprime{$'$} \def\cprime{$'$} \def\cprime{$'$} \def\cprime{$'$}
  \def\cdprime{$''$} \def\cprime{$'$} \def\cprime{$'$} \def\cprime{$'$}
  \def\cprime{$'$}
\providecommand{\bysame}{\leavevmode\hbox to3em{\hrulefill}\thinspace}
\providecommand{\MR}{\relax\ifhmode\unskip\space\fi MR }
\providecommand{\MRhref}[2]{%
  \href{http://www.ams.org/mathscinet-getitem?mr=#1}{#2}
}
\providecommand{\href}[2]{#2}

\end{document}